\documentclass[9pt,a4paper]{article}
\usepackage[T1]{fontenc}
\usepackage[francais,english]{babel}
\usepackage[applemac]{inputenc}
\usepackage{pb-diagram}
\usepackage[vmargin=3.5cm]{geometry}%
\usepackage{color}%
\usepackage{mdwlist}%
\usepackage{graphicx}%
\usepackage{float}%
\usepackage{makeidx}%
\usepackage{amsmath}%
\usepackage{amssymb}%
\usepackage{amsthm}%
\usepackage{amsfonts}%
\usepackage{mathrsfs}%
\usepackage{bbm}%
\usepackage[all,cmtip]{xy}%
\usepackage{rotate}%
\usepackage{yfonts}%
\usepackage{array}%
\newtheorem{Prop}{Proposition}[section]%
\newtheorem{Conj}[Prop]{Conjecture}%
\newtheorem{Question}[Prop]{Question}%
\newtheorem{Principle}[Prop]{Principle}
\newtheorem{TheoEnglish}[Prop]{Theorem}%
\newtheorem{DefEnglish}[Prop]{Definition}

\newtheorem{LemEnglish}[Prop]{Lemma}
\newtheorem{HypEnglish}[Prop]{Assumption}%
\newcommand{\A}{\mathbb A}%
\newcommand{\C}{\mathbb C}%
\newcommand{\Fp}{\mathbb F}%
\newcommand{\G}{\mathbf G}%

\newcommand{\Q}{\mathbb Q}%
\newcommand{\qp}{\mathbb Q_{p}}%
\newcommand{\R}{\mathbb R}%
\newcommand{\Z}{\mathbb Z}%
\newcommand{\zp}{\mathbb Z_{p}}%
\newcommand{\Fcal}{\mathcal F}%
\newcommand{\Ncal}{\mathcal N}%
\newcommand{\Ocal}{\mathcal O}%
\newcommand{\Kcal}{\mathcal K}%
\newcommand{\Rcal}{\mathcal R}%
\newcommand{\Scal}{\mathcal S}%
\newcommand{\Tcal}{\mathcal T}%
\newcommand{\Vcal}{\mathcal V}%
\newcommand{\Wcal}{\mathcal W}%
\newcommand{\Fcali}{\mathscr F}%

\newcommand{\Lcali}{\mathscr L}%
\newcommand{\aid}{\mathfrak a}%
\newcommand{\cid}{\mathfrak c}%
\newcommand{\pid}{\mathfrak p}%
\newcommand{\mgot}{\mathfrak m}%

\newcommand{\Lambdaf}{\mathbf{\Lambda}}
\newcommand{\GL}{\operatorname{GL}}%

\newcommand{\et}{\operatorname{et}}

\newcommand{\produit}[2]{\underset{#1}{\overset{#2}\prod}}%
\newcommand{\produittenseur}[2]{\underset{#1}{\overset{#2}\bigotimes}}%
\newcommand{\sommedirecte}[2]{\underset{#1}{\overset{#2}\bigoplus}}%
\newcommand{\applicationsimple}[3]{\begin{equation}%
\nonumber%
#1 :#2\longrightarrow #3%
\end{equation}}%
\newcommand{\application}[5]{\begin{eqnarray}%
\nonumber%
#1 :&#2&\longrightarrow #3\\
\nonumber%
&#4&\longmapsto #5
\end{eqnarray}}%
\newcommand{\suiteexacte}[5]{0\fleche#3\overset{#1}{\fleche}#4\overset{#2}{\fleche}#5\fleche0}

\newcommand{\limproj}[1]{\underset{\underset{#1}\longleftarrow}\lim}
\newcommand{\liminj}[1]{\underset{\underset{#1}\longrightarrow}\lim}

\newcommand{\coker}{\operatorname{coker}}
\newcommand{\Hom}{\operatorname{Hom}}

\newcommand{\isom}{\overset{\sim}{\longrightarrow}}

\newcommand{\plonge}{\hookrightarrow}

\newcommand{\rank}{\operatorname{rank}}%
\newcommand{\ord}{\operatorname{ord}}

\newcommand{\tenseur}{\otimes}
\newcommand{\Ltenseur}{\overset{\operatorname{L}}{\tenseur}}
\newcommand{\modulo}{\operatorname{ mod }}
\newcommand{\Spec}{\operatorname{Spec}}
\newcommand{\Id}{\operatorname{Id}}%
\newcommand{\Aut}{\operatorname{Aut}}%
\newcommand{\Frac}{\operatorname{Frac}}%
\newcommand{\Tate}{\operatorname{Ta}}%
\newcommand{\Cone}{\operatorname{Cone}}%
\newcommand{\fleche}{\longrightarrow}%
\newcommand{\croix}{^{\times}}%
\newcommand{\rhobar}{\bar{\rho}}%
\newcommand{\Hun}{H^{1}}

\newcommand{\Htilde}{\tilde{H}}

\newcommand{\RGamma}{\operatorname{R}\Gamma}%
\newcommand{\Det}{\operatorname{{D}et}}%
\newcommand{\RHom}{\operatorname{R}\Hom}%
\newcommand{\Fr}{\operatorname{Fr}}%
\newcommand{\Gal}{\operatorname{Gal}}

\newcommand{\Qbar}{\bar{\Q}}%
\newcommand{\Fbar}{\bar{F}}%
\newcommand{\Kbar}{\bar{K}}%
\newcommand{\Fpbar}{\bar{\mathbb F}}%
\newcommand{\Dbar}{\bar{D}}
\newcommand{\kbar}{\bar{k}}%
\newcommand{\s}{\sigma}%
\newcommand{\hgot}{\mathfrak h}%
\newcommand{\Hecke}{\mathbf{T}}%
\newcommand{\Eul}{\operatorname{Eul}}
\DeclareFontEncoding{OT2}{}{} 
\newcommand{\Nekovar}{Nekov\'a\v{r}}%
\newcommand{\cl}{\operatorname{cl}}
\numberwithin{equation}{subsection}%
\date{}
\begin{document}%
\title{Congruences of algebraic $L$-functions of motives}
\author{Olivier Fouquet and Jyoti Prakash Saha}%
\maketitle
\begin{abstract}
We develop a framework to investigate conjectures on congruences between the algebraic part of special values of $L$-functions of congruent motives. We show that algebraic local Euler factors satisfy precise interpolation properties in $p$-adic families of motives and that algebraic $p$-adic $L$-functions exist in quite large generality for $p$-adic families of automorphic motives. We formulate two conjectures refining (and correcting) the currently existing formulation of the Equivariant Tamagawa Number Conjecture with coefficients in Hecke algberas and pointing out the links between conjecture on special values and completed cohomology.
\end{abstract}
%
\selectlanguage{english}%

\newcommand{\hord}{\mathfrak h^{\ord}}%
\newcommand{\hdual}{\mathfrak h^{dual}}%
\newcommand{\matricetype}{\begin{pmatrix}\ a&b\\ c&d\end{pmatrix}}%
\newcommand{\Iw}{\operatorname{Iw}}%
\newcommand{\Hi}{\operatorname{Hi}}
\newcommand{\cyc}{\operatorname{cyc}}
\newcommand{\ab}{\operatorname{ab}}
\newcommand{\can}{\operatorname{can}}%
\newcommand{\Fitt}{\operatorname{Fitt}}%
\newcommand{\Tiwa}{\mathcal T_{\operatorname{Iw}}}%
\newcommand{\Af}{\operatorname{A}}%
\newcommand{\Dunzero}{D_{1,0}}%
\newcommand{\Uun}{U_{1}}%
\newcommand{\Uzero}{U_{0}}%
\newcommand{\Uundual}{U^{1}}%
\newcommand{\Uunun}{U^{1}_{1}}%
\newcommand{\Wdual}{\Wcal^{dual}}
\newcommand{\JunNps}{J_{1,0}(\Ncal, P^{s})}%
\newcommand{\Tatepord}{\Tate_{\pid}^{ord}}%
\newcommand{\Kum}{\operatorname{Kum}}%
\newcommand{\zcid}{z(\cid)}%
\newcommand{\kgtilde}{\tilde{\kappa}}%
\newcommand{\kiwa}{\varkappa}%
\newcommand{\kiwatilde}{\tilde{\varkappa}}%
\newcommand{\Hbar}{\bar{H}}%
\newcommand{\Tred}{T/\mgot T}%
\newcommand{\Riwa}{R_{\operatorname{Iw}}}%
\newcommand{\Kiwa}{\Kcal_{\operatorname{Iw}}}%
\newcommand{\Sp}{\mathbf{Sp}}%
\newcommand{\Aiwa}{\mathcal A_{\operatorname{Iw}}}%
\newcommand{\Viwa}{\mathcal V_{\operatorname{Iw}}}%
\newcommand{\pseudiso}{\overset{\centerdot}{\isom}}%
\newcommand{\pseudisom}{\overset{\approx}{\fleche}}%
\newcommand{\carac}{\operatorname{char}}%
\newcommand{\length}{\operatorname{length}}
\newcommand{\eord}{e^{\ord}}%
\newcommand{\eordm}{e^{\ord}_{\mgot}}%
\newcommand{\hordinfini}{\hord_{\infty}}%
\newcommand{\Mordinfini}{M^{\ord}_{\infty}}%
\newcommand{\hordm}{\hord_{\mgot}}%
\newcommand{\hminm}{\hgot^{min}_{\mgot}}%
\newcommand{\Mordm}{M^{\ord}_{\mgot}}%
\newcommand{\Mtwist}{M^{\tw}_{\mgot}}
\newcommand{\Xun}{X_{1}}%
\newcommand{\Xundual}{X^{1}}%
\newcommand{\Xunun}{X^{1}_{1}}%
\newcommand{\Xtw}{X^{\tw}}
\newcommand{\Inert}{\mathfrak{In}}%
\newcommand{\Tsp}{T_{\Sp}}%
\newcommand{\Asp}{A_{\Sp}}%
\newcommand{\Vsp}{V_{\Sp}}%
\newcommand{\SK}{\mathscr{S}}%
\newcommand{\Rord}{R^{\ord}}%
\newcommand{\per}{\operatorname{per}}
\newcommand{\z}{\mathbf{z}}
\newcommand{\zs}{\tilde{\mathbf{z}}}
\newcommand{\Ebarbar}{\bar{\bar{E}}}
\newcommand{\Grsym}{\mathfrak S}
\newcommand{\epsi}{\varepsilon}
\newcommand{\Fun}[2]{F^{{\mathbf{#1}}}_{#2}}
\newcommand{\tw}{\operatorname{tw}}
\newcommand{\ctf}{\operatorname{ctf}}
\newtheorem*{TheoA}{Theorem A}%
\newtheorem*{TheoB}{Theorem B}%
\newcommand{\isocan}{\overset{\can}{\simeq}}
\newcommand{\gen}{\operatorname{gen}}
\bigskip{\footnotesize%
  \textsc{Départment de Mathématiques, Bâtiment 425, Faculté des sciences d'Orsay Université Paris-Sud} \par  
  \textit{E-mail address}: \texttt{olivier.fouquet@math.u-psud.fr} \par
  \textit{Telephone number}: \texttt{+33169155729} \par
  \textit{Fax number}: \texttt{+33169156019}
  }
  
  \bigskip{\footnotesize%
  \textsc{Départment de Mathématiques, Bâtiment 425, Faculté des sciences d'Orsay Université Paris-Sud} \par  
  \textit{E-mail address}: \texttt{jyoti-prakash@math.u-psud.fr} \par
  \textit{Telephone number}: \texttt{+33169157237} \par
  \textit{Fax number}: \texttt{+33169156019}
  }
 \tableofcontents

\section{Introduction}
\subsection{Motivation}
In \cite{MazurValues}, the following remarkable question is asked about the special values of $L$-functions of rational eigenforms, a positive answer to which entails that congruent eigenforms have congruent special values.
\begin{Question}\label{Q1}
Can one expect [\dots] that the residue class, modulo an \emph{arbitrary} ideal $\aid$ in the Hecke algebra, of the "algebraic part" of the special value of the $L$-function attached to $\Gamma_{0}(N)$ (and its twists by Dirichlet characters $\chi$), is given by a "formula" determined by the representation of $\Gal(\Qbar/\Q)$ on the kernel of $\aid$ in the jacobian of $X_{0}(N)$?
\end{Question}
Note that B.Mazur had to resort twice to scare quotes in his formulation of the question as, at the time of his writing, the study of Galois representations with coefficients in Hecke algebras was barely in infancy and a conjectural theory of special values of $L$-functions was still lacking even for modular forms. After the fantastic progresses in the study of  $p$-adic families of automorphic Galois representations parametrized by Hecke algebras contained in and initiated by \cite{HidaInventionesOrdinary,WilesOrdinaryLambdaAdic,MazurDeformation,AshStevens,ColemanMazur,EmertonInterpolationEigenvalues,BellaicheChenevier,UrbanEigenvarieties} (among many other works) and in the light of the formulation of general conjectures for special values of $L$-functions of general motives in \cite{BlochKato,FontainePerrinRiou}, question \ref{Q1} can be precisely stated and investigated, even for more general motives.
\begin{Question}\label{Q2}
Suppose that $M_{1}$ and $M_{2}$ are two rational motives congruent modulo $p$ (in the sense that the mod $p$ residual representation attached to their $p$-adic étale realizations are isomorphic). Is it then true that the algebraic part of the special values of their $L$-functions (with Euler factors at $p$ removed) are congruent? More generally and more precisely, does the algebraic part of the special values of the $L$-functions in $p$-adic analytic families of motives vary analytically?
\end{Question}
Another bolder but natural question in this context is the following.
\begin{Question}\label{Q3}
Suppose that $M_{1}$ and $M_{2}$ are two rational motives congruent modulo $p$ and suppose that the $p$-part of the conjectures on special values of $L$-functions is true for $M_{1}$, is it then the case that $p$-part of the conjectures on special values of $L$-functions is true for $M_{2}$?
\end{Question}
If Mazur's question in its general form as in question \ref{Q2} admits a positive answer, then the truth of the $p$-part of the conjectures on special values of $L$-functions for $M_{1}$ implies that it is true modulo $p$ for $M_{2}$. What question \ref{Q3} raises is the question of whether this hypothesis is enough to ensure that the genuine full conjecture is true. Of course, if all the conjectures on special values of $L$-functions of motives are true, question \ref{Q3} tautologically admits a positive answer so an unconditional positive answer to this question can be seen as supporting evidence for the general conjectures. On the other hand, it should be noted that even if all conjectures on special values of individual motives were known, the answer to question \ref{Q2} would still not obviously be positive; and this is just as well since we show in sub-sections \ref{SubObstr} and \ref{SubExample} below that, as formulated, both questions \ref{Q1} and \ref{Q2} admit negative answers.

\subsection{Statement of results}
The aim of this manuscript is to propose a generalization and refinement of the conjectures of \cite{KatoHodgeIwasawa,KatoViaBdR} on special values of motives with coefficients encompassing the case of $p$-adic families of Galois representations parametrized by Hecke algebras and allowing for an investigation of questions \ref{Q2} and \ref{Q3}.

After reviewing a well-known obstruction relative to places of ramification for a too strong version of \ref{Q2} to always admit a positive answer, we show following \cite{TheseSaha} that this obstruction is the only one and that, in its absence, the far stronger statement that $p$-adic families of automorphic motives parametrized by Hecke algebras admit $p$-adic $L$-functions interpolating the algebraic part of the special values of $L$-functions at points for which the associated Galois representation is pure in the sense of the Weight-Monodromy Conjecture (so conjecturally at all classical points) is true. This result applies in particular unconditionally to $p$-adic families of eigenforms for $\GL_{2}$ over totally real fields or definite unitary groups over CM fields and it is in this form that we quote it in the introduction (see theorem \ref{TheoAuto} for a more general statement). 
\begin{TheoEnglish}\label{TheoIntro}
Let $\G$ be a definite unitary groups over an imaginary quadratic extension $F$ of $\Q$ or $\GL_{2}$ over a totally real field also denoted by $F$. Let $R_{\Sigma}$ be a local factor of the nearly-ordinary reduced Hecke algebra generated by operators outside $\Sigma$ attached to a residually absolutely irreducible representation. Let $\aid$ be a minimal prime of $R_{\Sigma}$. Then there exists a $p$-adic family $\Fcali$ of $G_{F,\Sigma}$-representations with coefficients in $R_{\Sigma}/\aid$ and an algebraic determinant $\Lcali_{\Sigma}(\Fcali)$ satisfying the interpolation property
\begin{equation}\nonumber
\Lcali_{\{p\}}(\Fcali)\tenseur_{R_{\Sigma}/\aid,\psi}\Ocal\isocan\Lcali_{\{p\}}(\Fcali\tenseur_{R_{\Sigma}/\aid,\psi}\Ocal)
\end{equation}
for all $\psi$ attached to classical automorphic representations of $\G(\A_{\Q}^{(\infty)})$.
\end{TheoEnglish}
We further show that a naïve formulation of the expected compatibilities of the statement of conjectures on special values of $L$-function with $p$-adic variation is incorrect.
\begin{TheoEnglish}\label{TheoFalse}
The generalized Iwasawa main conjecture of \cite[Section 3.2]{KatoViaBdR} is false as stated already for Hida families of eigencuspforms.
\end{TheoEnglish}
Our works suggests that  the correct formulation is closely linked with generalizations of the Local Langlands Correspondence to $p$-adic families parametrized by Hecke algebra and crucially involves the completed cohomology of \cite{EmertonInterpolationEigenvalues}. In that direction, we propose the conjectures \ref{ConjMain} and \ref{ConjCompleted} as corrected versions of \cite[Conjecture 3.2.2]{KatoViaBdR}. Though these conjectures are rather optimistic, we review in theorem \ref{TheoHecke} results of \cite{HeckeETNC} showing that they are true for rational eigencuspforms admitting Taylor-Wiles systems. As we show that the truth of conjecture \ref{ConjMain} implies a positive answer to question \ref{Q3}, this also settles question \ref{Q3} in the affirmative in that case. 

Going back one last time to Mazur's original question, we thus establish that though it admits a negative answer understood literally, a much stronger and much more beautiful statement-possibly the statement B.Mazur was hinting at behind the quotes-actually holds (at least under mild hypotheses on the residual representation). 

\paragraph{Acknowledgments:} The first name author thanks Barry Mazur and David Geraghty for encouraging him (in two very different ways) to think about the relation between completed cohomology and the trivializations of cohomology complexes as well as the organizers of the workshop \textit{$p$-adic aspects of modular forms} at IISER, Pune for the invitation to participate in this program. The second named author is grateful to the first named author for providing advice and encouragements while preparing \cite{TheseSaha}. He also thanks Yiwen Ding and Santosh Nadimpalli for useful discussions.

\subsection{General notations}
All rings are assumed to be commutative (and unital). If $R$ is a ring, we denote by $H^{q}_{\et}(R,-)$ the étale cohomology group $H^{q}(\Spec R,-)$. If $F$ is a field, we denote by $G_{F}$ the Galois group of a separable closure of $F$. If $F$ is a global field and $\Sigma$ is a finite set of places of $F$, we denote by $G_{F,\Sigma}$ the Galois group $\Gal(F_{\Sigma}/F)$ where $F_{\Sigma}$ is the maximal Galois extension of $F$ unramified outside $\Sigma$ and $\{v|\infty\}$. If $F$ is a number field and $v$ is a finite place of $F$, then $\Ocal_{F,v}$ is the unit ball of $F_{v}$, $\varpi_{v}$ is a choice of uniformizing parameter and $k_{v}$ is the residual field of $\Ocal_{F,v}$. The reciprocity law of local class field theory is normalized so that $\varpi_{v}$ is sent to (a choice of lit of) the geometric Frobenius morphism $\Fr(v)$. For all rational primes $\ell$, we fix an algebraic closure $\Qbar_{\ell}$ of $\Q_{\ell}$, an embedding of $\Qbar$ into $\Qbar_{\ell}$ and an identification $\iota_{\infty,\ell}:\C\simeq\Qbar_{\ell}$ extending $\Qbar\plonge\Qbar_{\ell}$. The non-trivial element of $\Gal(\C/\R)$ is denoted by $\tau$ and if $R$ is a ring in which $2$ is a unit and $M$ is an $R[\Gal(\C/\R)]$-module, then $M^{\pm}$ or $M^{\tau=\pm1}$ denotes as usual the eigenspace on which $\tau$ acts as $\pm 1$. We fix once and for all an odd prime $p$.

For $G$ a group, a $G$-representation $(T,\rho,R)$ is an $R$-module $T$ free of finite rank together with a continuous action
\begin{equation}\nonumber
\rho:G\fleche\Aut_{R}(T).
\end{equation}
Let $F$ be a number field and $R$ a complete local noetherian ring of residual characteristic $p$. If $\Sigma$ is a finite set of places of $F$ containing $\{v|p\}$ and if $(T,\rho,R)$ is a $G_{F,\Sigma}$-representation, then $T$ can also be considered as an étale sheaf over $\Spec\Ocal_{F}[1/\Sigma]$ and the continuous group cohomology $H^{i}(G_{F,\Sigma},T)$ coincides with the étale cohomology $H^{i}_{\et}(\Ocal_{F}[1/\Sigma],T)$.

We refer to the appendix for notations and conventions regarding the determinant functor and complexes of cohomology with local conditions.

A notational difficulty of this manuscript is that two of its main objects of study are the functors from the category of perfect complexes of $R$-modules to the category of graded invertible $R$-modules of \cite{MumfordKnudsen} and the polynomial laws on group-algebras of \cite{ChenevierDeterminant}. These two families of objects are both called determinants. In order to help distinguish them, the determinant functor is always written with a capitalized initial.

\section{Conjecture on special values and $p$-adic families of motives}
\subsection{Review of the Tamagawa Number Conjecture}
In sub-section \ref{SubTNC}, we review the formalism of the Tamagawa Number Conjecture for pure motives over $\Q$ following \cite{KatoHodgeIwasawa,FontainePerrinRiou}. Inspired by Grothendieck's trace formula and its use in the expression of the $L$-function of a scheme over a field of positive characteristic (as recalled in sub-section \ref{SubFinite}), we then recall in sub-section \ref{SubComp} the refinement of \cite{KatoHodgeIwasawa} describing the behavior of the Tamagawa Number Conjecture under proper base change and outline the conjectural description given in \cite{KatoViaBdR} of the variation of special values of $p$-adic étale sheaves on $\Spec\Z[1/p]$ with change of ring of coefficients.
\subsubsection{The case of rational motives}\label{SubTNC}
 Let $M$ be a rank $n$ rational motive pure of weight $w-2i$, which for concreteness we can take to be a direct summand of the shifted absolute cohomology $h^{w}(X/\Q)(i)$ of a proper smooth scheme $X$ over $\Q$. Let $S$ be a finite set of primes containing the fixed odd prime $p$. The $p$-adic $G_{\Q}$-representation $(H_{\et}^{w}(X\times_{\Q}\Qbar,\qp)(i),\rho,\qp)$ arising from the $p$-adic étale realization of $M$ is denoted by $M_{\et,p}$. We also view $M_{\et,p}$ as a smooth étale sheaf over $\Spec\Z[1/p]$. In this manuscript, we always consider as an integral part of the definition of a pure motive the fact that $M_{\et,p}$ satisfies the Weight-Monodromy Conjecture of \cite[Conjecture 3.9]{IllusieMonodromie}.
 \begin{HypEnglish}\label{HypWMC}
Let $\ell\neq p$ be a finite prime and let $\s_{\ell}\in G_{\Q_{\ell}}$ be a lift of $\Fr(\ell)$. The eigenvalues of $\s_{\ell}$ acting on the the $j$-th graded part of the monodromy filtration of the $p$-adic $G_{\Q_{\ell}}$-representation $M_{\et,p}$ are Weil numbers of weight $w-2i+j$.
 \end{HypEnglish}
 
For all $\ell\notin S$, the Euler factor
\begin{equation}\nonumber
\Eul_{\ell}(M,X)=\det(1-\Fr(\ell)X|M_{\et,p}^{I_{\ell}})
\end{equation}
is conjectured to lie in $\Q[X]$ (and to be independent of the choice of $p$, as the notation suggests).  Under this conjecture, the partial $L$-function $L_{S}(M,s)$ is the complex $L$-function defined by the partial Euler product 
\begin{equation}\nonumber
L_{S}(M,s)=\produit{\ell\notin S}{}\frac{1}{\Eul_{\ell}(M,p^{-s})}=\produit{\ell\notin S}{}\frac{1}{\det(1-\Fr(\ell)p^{-s}|M_{\et,p}^{I_{\ell}})}.
\end{equation}
for $s$ sufficiently large. It is conjectured to admit a meromorphic continuation to $\C$ satisfying a function equation relating $L(M,s)$ and $L(M^{*}(1),-s)$. Denote by $L^{*}_{S}(M^{*}(1),0)$ the first non-zero term in the Taylor expansion of $L_{S}(M,s)$ about zero. The analytic special value of the $L$-function of $M^{*}(1)$ at zero is the complex number $L^{*}_{S}(M^{*}(1),0)$. 

The complex $\RGamma_{\et}(\Z[1/S],M_{\et,p})$ is a bounded (hence perfect) complex of $\qp$-vector spaces and, more generally, for all $G_{\Q}$-stable $\zp$-lattice $T$ inside $M_{\et,p}$, the complex $\RGamma_{\et}(\Z[1/S],T)$ is a perfect complex of $\zp$-modules. Seen as complexes placed in degree 0, $M_{\et,p}^{+}$ and $T^{+}$ are also perfect complexes (of $\qp$-vector spaces and $\zp$-modules respectively). Hence, there exists free modules of rank 1
\begin{equation}\nonumber
\Det^{-1}_{\zp}\RGamma_{\et}(\Z[1/S],T)\tenseur_{\zp}\Det^{-1}_{\zp}(T^{+})
\subset\Det^{-1}_{\qp}\RGamma_{\et}(\Z[1/S],M_{\et,p})\tenseur_{\qp}\Det^{-1}_{\qp}(M_{\et,p}^{+})
\end{equation} 
over $\qp$ and $\zp$ respectively. The $\zp$-lattice defined in the right-hand side by the image of the left-hand side under tensor product with $\qp$ is independent of the choice of $T$ by Tate's formula \cite[Theorem 2.2]{TateBSD}.

The following conjecture, called the $p$-part of the Tamagawa Number Conjecture (henceforth TNC), is due to Bloch-Kato and is given here in the formulation of \cite{FontaineValeursSpeciales,KatoHodgeIwasawa,FontainePerrinRiou}.
\begin{Conj}[Tamagawa Number Conjecture]\label{ConjTNC}
For all finite set of primes $S\supset\{p\}$, there exists a fundamental line $\Delta_{S}(M/\Q)$ which is a $\Q$-vector space of dimension 1 equipped with a canonical isomorphism 
\begin{equation}\nonumber
\per_{\C}:\Delta_{S}(M/\Q)\tenseur_{\Q}\C\isocan\C
\end{equation}
of complex periods and with a canonical isomorphism 
\begin{equation}\nonumber
\per_{p}:\Delta_{S}(M/\Q)\tenseur_{\Q}\qp\isocan\Det^{-1}_{\qp}\RGamma_{\et}(\Z[1/S],M_{\et,p})\tenseur_{\qp}\Det^{-1}_{\qp}(M_{\et,p}^{+})
\end{equation}
of $p$-adic periods as well as a motivic element $\z_{S}(M/\Q)$ which is a basis of $\Delta_{S}(M/\Q)$ verifying the following properties.
\begin{enumerate}
\item The image of $\z_{S}(M/\Q)\tenseur1$ through $\per_{\C}$ is equal to $L^{*}_{S}(M^{*}(1),0)$.
\item The image of the $\zp$-lattice $\zp\cdot(\z_{S}(M/\Q)\tenseur1)$ through $\per_{p}$ is equal to 
\begin{equation}\nonumber
\Det^{-1}_{\zp}\RGamma_{\et}(\Z[1/S],T)\tenseur_{\zp}\Det^{-1}_{\zp}(T^{+})
\end{equation}
for any choice of $G_{\Q}$-stable $\zp$-lattice $T$ inside $M_{\et,p}$.
\end{enumerate}
\end{Conj}
As the notation suggests, the conjectural definition of $\per_{\C}$ and $\per_{p}$ is independent of the choice of the set $S$.

\subsubsection{An interlude on function fields}\label{SubFinite}
Let $f:X\fleche\Fp_{q}$ be a finite separated scheme over a field of positive characteristic $\ell\neq p$. Denote $X\times_{\Fp_{q}}{\Fpbar_{q}}$ by $\bar{X}$ and let $\Lambda$ be a complete local noetherian ring of residual characteristic $p$. Let $\Fcali$ be a complex of smooth sheaves of $\Lambda$-modules on $X$ (that is to say a complex of projective systems of locally constant sheaves of torsion $\Lambda$-modules with transition maps compatible with reduction modulo $\mgot^{n}_{\Lambda}$). Denote by 
\begin{equation}\nonumber
\Fcali^{*}=\RHom_{\Lambda}(\Fcali,f^{!}\Lambda).
\end{equation}
its dual complex of sheaves. The $L$-function $L(X,\Fcali^{*},t)$ of $\Fcali^{*}$
\begin{equation}\nonumber
L(X,\Fcali^{*},t)=\produit{x\in X}{}\frac{1}{\det(1-\Fr_{x}t^{\deg(x)},\Fcali^{*})}
\end{equation}
is as in \cite[Expose III]{DixExposes} the product of the determinants of the Frobenius morphisms at the stalks of the geometric points of $X$. By the Grothendieck's trace formula of \cite[Théorème 3.2, Rapport]{SGA41/2}, $L(X,\Fcali^{*},t)$ is computed by the action of $1-\Fr$ on $\RGamma_{c}(X\times_{\Fp_{q}}\Fpbar_{q},\Fcali^{*})$ as 
\begin{equation}\label{EqGro}
L(X/\Fp_{q},\Fcali^{*},t)=\produit{i\in\Z}{}\det\left(1-t\Fr,H^{i}_{c}(\bar{X},\Fcali^{*})\right)^{(-1)^{i+1}}
\end{equation}
or equivalently by Poincaré's duality by the action on $\RGamma_{\et}(X\times_{\Fp_{q}}\Fpbar_{q},\Fcali)$ as
\begin{align}\nonumber
L(X/\Fp_{q},\Fcali^{*},t)&=\det\left(1-t\Fr,\RGamma_{c}(\bar{X},\Fcali^{*})\right)^{-1}\\\nonumber
&=\det\left(1-t\Fr,\RGamma_{\et}(\bar{X},\Fcali)\right)^{-1}.
\end{align}
The complexes $\RGamma_{\et}(X,\Fcali)$ and $\RGamma_{\et}(X\times_{\Fp_{q}}\Fpbar_{q},\Fcali)$ fit in the distinguished triangle
\begin{equation}\nonumber
\xymatrix{
\RGamma(X,\Fcali)\ar[rr]&&\RGamma(\bar{X},\Fcali)\ar^{1-\Fr}[ld]\\
&\RGamma(\bar{X},\Fcali)\ar[lu]&
}
\end{equation}
inducing canonical trivializations
\begin{equation}\nonumber
\Det^{-1}_{\Lambda}\RGamma_{\et}(X,\Fcali)\isocan\Det_{\Lambda}\RGamma_{\et}(\bar{X},\Fcali)\tenseur\Det_{\Lambda}^{-1}\RGamma_{\et}(\bar{X},\Fcali)\isocan\Lambda.
\end{equation}
 The spectral sequence
\begin{equation}\nonumber
E^{i,j}=H^{i}(G_{\Fp_{q}},H^{j}(\bar{X},\Fcali))
\end{equation}
arising from composition of the functors $\Gamma_{\et}(\bar{X},-)$ and $\Gamma(G_{\Fp_{q}},-)$ and converging to $H^{i+j}(X,\Fcali)$ induces short exact sequences
\begin{equation}\label{EqSuitesCourtes}
\suiteexacte{}{}{\coker(1-\Fr)|_{H^{i-1}_{\et}(\bar{X},\Fcali)}}{H^{i}(X,\Fcali)}{\ker(1-\Fr)|_{H^{i}_{\et}(\bar{X},\Fcali)}}.
\end{equation}
Under the standard hypothesis that $1-\Fr$ acts semi-simply on $H^{i}_{\et}(\bar{X},\Fcali)$, there is an isomorphism between $\ker(1-\Fr)|_{H^{i}_{\et}(\bar{X},\Fcali)}$ and $\coker(1-\Fr)|_{H^{i}_{\et}(\bar{X},\Fcali)}$. Hence, the short exact sequences \eqref{EqSuitesCourtes} define an isomorphism 
\begin{equation}\nonumber
\psi:\Det^{-1}_{\Lambda}\RGamma_{\et}(X,\Fcali)\overset{}{\simeq}\Lambda.
\end{equation}
Equation \eqref{EqGro} can then be reformulated as stating that the pre-image of $L^{*}(X/\Fp_{q},\Fcali^{*},1)$ through $\psi$ in
\begin{equation}\label{EqFinite}
\xymatrix{
\Det^{-1}_{\Lambda}\RGamma_{\et}(X,\Fcali)\ar[r]^{\psi}\ar[d]&\Lambda\ni L^{*}(X/\Fp_{q},\Fcali^{*},1)\\
\Det^{-1}_{\Lambda}\RGamma_{\et}(\bar{X},\Fcali)\tenseur_{\Lambda}\Det_{\Lambda}\RGamma_{\et}(\bar{X},\Fcali)\ar[d]\\
\Lambda\ni 1&
}
\end{equation}
is the basis of $\Det^{-1}_{\Lambda}\RGamma_{\et}(X,\Fcali)$ sent to $1\in\Lambda$ by the vertical arrows. Denote this basis of $\Det^{-1}_{\Lambda}\RGamma_{\et}(X,\Fcali)$ by $\z_{\Lambda}(X,\Fcali)$.

The properties of the pair $(\z_{\Lambda}(X,\Fcali),\Det^{-1}_{\Lambda}\RGamma_{\et}(X,\Fcali))$ are highly reminiscent of the properties of the pair $(\z_{S}(M/\Q),\Delta_{S}(M/\Q))$ of conjecture \ref{ConjTNC}. The Tamagawa Number Conjecture can thus be considered as an analogue of Grothendieck's trace formula for motives over number fields. As it was stated in sub-section \ref{SubTNC}, however, the TNC falls short of being a completely satisfying analogy. Firstly, Grothendieck's trace formula and more generally the construction of \eqref{EqGro} is valid for a complexes of sheaves of $\Lambda$-modules whereas the TNC is formulated only for sheaves on $\Spec\Z[1/S]$ arising from a motive. Secondly, Grothendieck's trace formula is compatible with proper base change via the morphism $f:X\fleche Y$ in the sense that 
\begin{equation}\label{EqCompProper}
\z_{\Lambda}(X,\Fcali)=\z_{\Lambda}(Y,Rf_{*}\Fcali)
\end{equation}
in
\begin{equation}\nonumber
\Det^{-1}_{\Lambda}\RGamma_{\et}(X,\Fcali)=\Det^{-1}_{\Lambda}(Y,Rf_{*}\Fcali).
\end{equation}
Thirdly, the pair $\z_{\Lambda}(X,\Fcali)\in\Det^{-1}_{\Lambda}\RGamma_{\et}(X,\Fcali)$ is compatible with derived change of rings of coefficients in the sense that  
\begin{equation}\label{EqCompChange}
\z_{\Lambda}(X,\Fcali)\tenseur1=\z_{\Lambda'}(X,\Fcali\Ltenseur\Lambda')
\end{equation}
in
\begin{equation}\nonumber
\Det^{-1}_{\Lambda'}\RGamma_{\et}(X,\Fcali)\Ltenseur_{\Lambda}\Lambda'=\Det^{-1}_{\Lambda'}(Y,\Fcali\Ltenseur_{\Lambda}\Lambda').
\end{equation}
The TNC \textit{a priori} satisfies none of these supplementary properties.
\subsubsection{Compatibility with $\Spec\Ocal_{K}\fleche\Spec\Z$}\label{SubComp}
Following the model of $L$-function of schemes over finite fields, we recall the classical generalization due to \cite{KatoHodgeIwasawa} (see also \cite{BurnsFlachMotivic}) of conjecture \ref{ConjTNC} to incorporate the compatibility with proper base change \eqref{EqCompProper}.

Let $K$ be a finite abelian extension of $\Q$ with Galois group $G$ and ring of integers $\Ocal_{K}$. To $K$ is attached a rational motive $h^{0}(\Spec K)$ whose Betti, de Rham and $p$-adic étale realizations are respectively $\Q[G]$, $K$ with the filtration $F^{0}K=K$ and $F^{1}K=0$ and natural action of $G$ and $\qp[G]$. If as in sub-section \ref{SubTNC} we denote by $M$ a motive over $\Q$ and by $S$ a finite set of rational primes, there exists a rational motive $M_{K}=M\otimes h^{0}(\Spec K)$. The local Euler factor $\Eul_{\ell}(M_{K},X)$ for $\ell\neq p$ then verifies
\begin{equation}\nonumber
\Eul_{\ell}(M_{K},X)=\det(1-\Fr(\ell)X|M^{I_{\ell}}_{K,\et,p})=\produit{\chi\in\hat{G}}{}\det\left(1-\Fr(\ell)X|(M\tenseur\Qbar_{p}(\chi))^{I_{\ell}}\right).
\end{equation}
We assume the conjecture that it belongs to $\Q[G][X]$. The $S$-partial $L$-function
\begin{equation}\nonumber
L_{S}(M_{K},s)=\produit{\ell\notin S}{}\frac{1}{\Eul_{\ell}(M_{K},p^{-s})}\in\C[G]^{\C}
\end{equation}
is the product of the Euler factors for $\ell\notin S$.

The following conjecture, due to \cite{KatoHodgeIwasawa}, expresses the compatibility of conjecture \ref{ConjTNC} with the proper base change map $\Spec\Ocal_{K}\fleche\Spec\Z$. We henceforth refer to it as the $p$-part of the Equivariant Tamagawa Number Conjecture for the motive $M$, the extension $K$ and the set $S$ or, more commonly and more simply, as the ETNC.
\begin{Conj}[Equivariant Tamagawa Number Conjecture]\label{ConjETNC}
For all finite abelian extension $K$ of $\Q$ with Galois group $G$ and for all finite set of primes $S\supset\{p\}$, there exists an equivariant fundamental line $\Delta_{K,S}(M/\Q)$ which is a $\Q[G]$-vector space of dimension 1 equipped with a canonical isomorphism 
\begin{equation}\nonumber
\per_{\C[G]}:\Delta_{K,S}(M/\Q)\tenseur_{\Q[G]}\C\isocan\C[G]
\end{equation}
of complex periods and with a canonical isomorphism 
\begin{equation}\nonumber
\per_{p}:\Delta_{K,S}(M/\Q)\tenseur_{\Q[G]}\qp\isocan\Det^{-1}_{\qp[G]}\RGamma_{\et}(\Ocal_{K}[1/S],M_{\et,p})\tenseur_{\qp[G]}\Det^{-1}_{\qp[G]}(M_{K,\et,p}^{+})
\end{equation}
of $p$-adic periods as well as an equivariant motivic element $\z_{K,S}(M/\Q)$ which is a basis of $\Delta_{K,S}(M/\Q)$ verifying the following properties.
\begin{enumerate}
\item The image of $\z_{K,S}(M/\Q)\tenseur1$ through $\per_{\C}$ is equal to $L^{*}_{S}(M_{K}^{*}(1),0)$.
\item The image of the $\zp[G]$-lattice $\zp[G]\cdot(\z_{K,S}(M/\Q)\tenseur1)$ through $\per_{p}$ is equal to 
\begin{equation}\nonumber
\Det^{-1}_{\zp[G]}\RGamma_{\et}(\Ocal_{K}[1/S],T)\tenseur_{\zp[G]}\Det^{-1}_{\zp[G]}((T\tenseur_{\zp}\zp[G])^{+})
\end{equation}
for any choice of $G_{\Q}$-stable $\zp$-lattice $T$ inside $M_{\et,p}$.
\end{enumerate}
\end{Conj}
We record the following important fact. Suppose $L/\Q$ is a finite abelian extension, that $\Q\subset K\subset L$ is a finite sub-extension, that $S_{K}$ is a finite set of primes containing exactly $p$ and all the primes of $\Q$ ramifying in $K$ and that $S_{L}$ is a finite set of primes containing exactly $p$ and all the primes of $\Q$ ramifying in $L$. If conjecture \ref{ConjETNC} is true for the triplets $(M,K,S_{K})$ and $(M,L,S_{L})$, then the trace from $\C[\Gal(L/\Q)]$ to $\C[\Gal(K/\Q)]$ sends 
$L_{S_{L}}(M_{L}^{*}(1),s)$ to
\begin{equation}\nonumber
L_{S_{K}}(M_{K}^{*}(1),s)\produit{\ell\in S_{L}\backslash S_{K}}{}\Eul_{\ell}(M^{*}(1),p^{-s})
\end{equation}
and thus likewise the trace from $\qp[\Gal(L/\Q)]$ to $\qp[\Gal(K/\Q)]$ sends $\z_{L,S_{L}}(M)\tenseur1$ to 
\begin{equation}\nonumber
(\z_{K,S_{K}}(M)\tenseur1)\produit{\ell\in S_{L}\backslash S_{K}}{}\Eul_{\ell}(M^{*}(1),1).
\end{equation}
Hence the system of classes $\{\z_{L,S_{L}}(M)\}_{L}$ for $\Q\subset L\subset\Q^{ab}$ satisfies the relations defining an Euler system. Because an analogous result plays an important role further below, we restate this point: the compatibility of the TNC with proper base change morphisms implies that motivic zeta elements form an Euler system.

The collection of ETNC for the abelian $p$-extensions of $\Q$ imply a generalization of conjecture \ref{ConjTNC} with coefficients in rings of larger dimension and is thus a prototype of the phenomena we discuss in subsequent sub-sections. For $n\geq1$, let $\Q_{n}/\Q$ be the sub-extension of $\Q(\zeta_{p^{n}})$ with Galois group isomorphic to $\Z/p^{n-1}\Z$ and let $\Ocal_{n}$ be its ring of integers. Let $\Q_{\infty}/\Q$ be the unique $\zp$-extension of $\Q$, that is to say the union of the $\Q_{n}$. Let $\Gamma$ be the Galois group $\Gal(\Q_{\infty}/\Q)$ and let $\Lambda$ denote as usual the 2-dimensional regular local ring equal to the completed group algebra $\zp[[\Gamma]]$. The ring $\Lambda$ is endowed with a $G_{\Q}$-action through the surjection of $G_{\Q}$ onto $\Gamma$ followed by the inclusion of $\Gamma$ in $\Lambda\croix$. The cyclotomic deformation $(T\tenseur\Lambda,\rho_{\Lambda},\Lambda)$ is the $G_{\Q}$-representation $T\tenseur_{\zp}\Lambda$ with $G_{\Q}$-action on both sides of the tensor product. Likewise, let $\mathbf{\Lambda}$ be the completed group-algebra $\zp[[\Gal(\Q(\zeta_{p^{\infty}})/\Q)]]$ and let $T\tenseur\mathbf{\Lambda}$ be the $G_{\Q}$-representation with action on both sides of the tensor product. Assuming the truth of conjectures \ref{ConjETNC} for $M$ and $\Q(\zeta_{p^{n}})/\Q$ for all $n\geq0$ yields an element $\z_{\mathbf{\Lambda},S}(M/\Q)$ which is a basis of $\Det^{-1}_{\Lambdaf}\RGamma_{\et}(\Z[1/p],T\tenseur_{\zp}\Lambdaf)\tenseur_{\Lambdaf}\Det^{-1}_{\Lambdaf}(T\tenseur_{\zp}\Lambdaf)^{+}$ and such that 
\begin{equation}\label{EqCompChangeMotive}
\z_{\Lambdaf,S}(M/\Q)\tenseur1=\z_{K,S}(M/\Q)
\end{equation}
through the canonical isomorphism between
\begin{equation}\nonumber
\left(\Det^{-1}_{\Lambdaf}\RGamma_{\et}(\Z[1/p],T\tenseur_{\zp}\Lambdaf)\tenseur_{\Lambdaf}\Det^{-1}_{\Lambdaf}(T\tenseur_{\zp}\Lambdaf)^{+}\right)\tenseur_{\Lambdaf}\zp[G]
\end{equation}
and 
\begin{equation}\nonumber
\Det^{-1}_{\zp[G]}\RGamma_{\et}(\Z[1/p],T\tenseur_{\zp}\zp[G])\tenseur_{\zp[G]}\Det^{-1}_{\zp[G]}(T\tenseur_{\zp}\zp[G])^{+}
\end{equation}
for all finite sub-extension $\Q\subset K\subset\Q(\zeta_{p^{\infty}})$ with Galois group $G$. Hence, a statement analogous to \eqref{EqCompChange} holds for the pair of motivic zeta elements $\z_{\Lambdaf,S}(M/\Q)$ and $\z_{K,S}(M/\Q)$ with coefficients in $\Lambdaf$ and $\zp[G]$ respectively.

\subsubsection{Compatibility with change of coefficients}\label{SubChange}
The compatibility of conjecture \ref{ConjTNC} (or for that matter of conjecture \ref{ConjETNC}) with change of ring of coefficients in the sense of equations \eqref{EqCompChange} and \eqref{EqCompChangeMotive} for arbitrary pairs of rings is much harder to even formulate. A tentative statement for complete local noetherian rings with residual field of characteristic $p$ would be as follows. 

Suppose that $\Fcali$ is a smooth étale sheaf (or even a perfect complex of smooth étale sheaves) of $\Lambda$-modules on $\Spec\Z[1/p]$ (here $\Lambda$ is again as in section \ref{SubFinite} a complete local noetherian ring). Suppose that for some $x\in\Hom(\Lambda,\Qbar_{p})$, the base-changed sheaf $\Fcali_{x}=\Fcali\tenseur_{\Lambda,x}\Qbar_{p}$ arises as the $G_{\Q}$-representation attached to a pure motive $M_{x}$. We would then define the fundamental line of $\Fcali$ as 
\begin{equation}\nonumber
\Delta_{\Lambda}(\Spec\Z[1/S],\Fcali)=\Det^{-1}_{\Lambda}\RGamma_{\et}(\Z[1/S],\Fcali)\tenseur_{\Lambda}\Det^{-1}_{\Lambda}\Fcali^{+}
\end{equation}
and conjecture the existence of a basis $\z_{\Lambda}(\Spec\Z[1/S],\Fcali)$ of $\Delta_{\Lambda}(\Spec\Z[1/S],\Fcali)$ which satisfies the following properties.
\begin{enumerate}
\item If $f$ is the structural morphism $\Spec\Ocal_{K}[1/S]\fleche\Spec\Z[1/S]$ for $K$ a finite abelian extension of $\Q$, then 
\begin{equation}\label{EqCompProperF}
\z_{\Lambda}(\Spec\Ocal_{K}[1/S],\Fcali)=\z_{\Lambda}(\Spec\Z[1/S],Rf_{*}\Fcali)
\end{equation}
in
\begin{equation}\nonumber
\Delta_{\Lambda}(\Spec\Ocal_{K}[1/S],\Fcali)=\Det^{-1}_{\Lambda}(\Spec\Z[1/S],Rf_{*}\Fcali).
\end{equation}
\item If $\pi:\Lambda\fleche\Lambda'$ is a local morphism of completed local noetherian rings as above, the pair $\z_{\Lambda}(\Spec\Z[1/S],\Fcali)\in\Delta_{\Lambda}(\Spec\Z[1/S],\Fcali)$ is compatible with derived change of rings of coefficients in the sense that  
\begin{equation}\label{EqCompChangeF}
\z_{\Lambda}(\Spec\Z[1/S],\Fcali)\tenseur1=\z_{\Lambda'}(\Spec\Z[1/S],\Fcali\Ltenseur\Lambda')
\end{equation}
in
\begin{equation}\nonumber
\Delta_{\Lambda}\tenseur_{\Lambda}\Lambda'\isocan\Delta_{\Lambda'}(\Spec\Z[1/S],\Fcali\Ltenseur_{\Lambda}\Lambda').\end{equation}
\item The zeta element $\z_{\zp}(\Spec\Z[1/S],\Fcali_{x})$ coincides with the zeta element $\z_{S}(M_{x}/\Q)\tenseur1$.
\end{enumerate}
This formulation is the one chosen in \cite{KatoViaBdR}. 

Though it has the merit of finally putting conjectures for motives on number fields on an equal footing with schemes over finite fields,\footnote{And, more trivially, though \cite[Conjecture 3.2.2]{KatoViaBdR} has been the single most influential mathematical statement on the thought process of the first named author.} there are reasons to believe it is not completely satisfying. To start with, actually existing $p$-adic families of motives are rarely known to satisfy the very first hypotheses required to formulate the conjecture: outside the case of families of twists by $\zp$-extensions (which strictly speaking reflect the compatibility \eqref{EqCompProperF} rather than the compatibility \eqref{EqCompChangeF}), it seems that no $p$-adic families of motives are known to yield smooth étale sheaves over $\Spec\Z[1/p]$. As can be seen already for the étale cohomology of modular curves, the problem is that there seems to be no reason for the $I_{\ell}$-invariants of $\Fcali$ to form a perfect complex of étale sheaves on $\Spec\Z_{\ell}$ at places of ramification when the ring of coefficients is sufficiently general, and in particular when it is taken to be the Hecke algebra. Besides, and more worryingly, there are reasons to believe that the conjecture outlined above is actually not quite correct; a point to which we return in sub-section \ref{SubHecke} below. 

\subsection{Congruences between motives}
As the formalism of perfect complexes of étale sheaves on $\Spec\Z[1/p]$ is not known to apply to actually existing $p$-adic family of motives, we introduce in this sub-section a general framework in which to carry the study of congruences between special values of $L$-functions of motives and the compatibility of the TNC with base change of ring of coefficients.
\subsubsection{Determinants} 
We recall from \cite{ChenevierDeterminant} that an $n$-dimensional $A$-valued determinant on an $A$-algebra $R$ is a multiplicative $A$-polynomial law
\begin{equation}\nonumber
D:R\fleche A
\end{equation}
homogeneous of degree $n$ (see \cite[Section 1.5]{ChenevierDeterminant}). An $A$-valued determinant on $A[G]$ is also called an $A$-valued determinant on $G$. Because an $A$-valued determinant on $R$ is extended in an obvious way in an $A[X]$-valued determinant on $R[X]$, we do not distinguish between these two notions. When there exists an $A$-valued determinant on a group $G$, the characteristic polynomial $\chi(D,g)\in A[X]$ of an element $g\in G$ is defined to be $D(1-Xg)$. Determinants, and hence characteristic polynomials, are compatible with base change of ring of coefficients.

Let $M$ be a rational motive as in the previous sub-section. By compactness, $(M_{\et,p},\rho,\qp)$ admits a free $\zp$-lattice stable under the action of $G_{\Q}$. The map
\applicationsimple{D}{\zp[G_{\Q}]}{\zp}
sending ${g\in G_{\Q}}$ to ${\det\rho(g)}$ is a determinant which we call the determinant attached to $M$ or to $M_{\et,p}$. Composing $D$ with reduction modulo $p$ and injection into $\bar{\Fp}_{p}$ yields an $n$-dimensional determinant
\begin{equation}\nonumber
\bar{D}:\bar{\Fp}_{p}[G_{\Q}]\fleche\bar{\Fp}_{p}
\end{equation}
and hence, by \cite[Theorem 2.12]{ChenevierDeterminant}, a semi-simple $n$-dimensional residual representation
\begin{equation}\nonumber
\rhobar:G_{\Q}\fleche\Aut_{\Fpbar_{p}}(\bar{M}_{\et,p})
\end{equation}
unique up to isomorphism which we call the residual representation attached to $M$ (or $M_{\et,p}$). Two motives $M_{1}$ and $M_{2}$ are said to be congruent modulo $p$ if their residual representations are isomorphic.

More generally, let $R$ be a complete local noetherian ring with maximal ideal $\mgot$ and finite residual field $k=R/\mgot$ of characteristic $p$. Let $\Sigma$ be a finite set of finite primes containing $p$. Let
\applicationsimple{D}{R[G_{\Q,\Sigma}]}{R}
be a degree $n$ determinant. A specialization of $R$ is a morphism $\psi:R\fleche S$ of local $\zp$-algebras. If $\psi:R\fleche S$ is a specialization of $R$, we denote by $D_{\psi}$ the determinant $\psi\circ D$ and call it the specialization of $D$ at $\psi$.
If $\psi$ has values in an algebraically closed field $\Kbar$, then there exists an $n$-dimensional semi-simple $G_{\Q}$-representation $\rho_{\psi}$ such that $D_{\psi}=\psi\circ D$. The $G_{\Q,\Sigma}$-representation $\rhobar$ attached in this way to the ring morphism
\begin{equation}\nonumber
R\fleche k\plonge\kbar
\end{equation}
is called the residual representation attached to $D$. Assume there exists a set of ring morphisms $\Hom^{\cl}(R,\Qbar_{p})$ such that for all $\psi\in\Hom^{\cl}(R,\Qbar_{p})$, the $G_{\Q}$-representation $(M_{\psi},\rho_{\psi},\Qbar_{p})$ is the $G_{\Q}$-representation attached to the $p$-adic étale realization of a pure motive $M_{\psi}$. We then call the set
\begin{equation}\nonumber
\mathscr{F}=\{M_{\psi}|\psi\in\Hom(R,\Qbar_{p})\}
\end{equation}
the $p$-adic family of $G_{\Q}$-representations parametrized by $\Spec R$ or with coefficients in $R$ attached to $D$. Though we typically have in mind the case of a ring $R$ of large Krull dimension, we note that a single $p$-adic $G_{\Q}$-representation is a family in this sense. In a slight abuse of terminology, we say that a motive $M$ belongs to the family $\Fcali$ if its $p$-adic étale realization (base-changed to $\Qbar_{p}$) belongs to $\Fcali$ and consequently sometimes refer to $\Fcali$ as a $p$-adic family of motives (this is reasonably harmless when $\Hom^{\cl}(R,\Qbar_{p})$ is Zariski dense in $\Hom(R,\Qbar_{p})$ but is potentially misleading otherwise). 

If $\psi:R\fleche S$ is a specialization of $R$ (not necessarily attached to a motive $M_{\psi}$), we denote by $\Fcali_{\psi}$ the $p$-adic family attached to the determinant $D_{\psi}$.

If two motives $M_{1}$ and $M_{2}$ belong to the same $p$-adic family $\Fcali$ with coefficients in $R$ and attached to $D$, the determinants $D_{1}$ and $D_{2}$ attached to them are specializations of $D$ and so their residual representations are identical up to isomorphism. Thus, $M_{1}$ and $M_{2}$ are congruent modulo $p$. Conversely, suppose $M_{1}$ and $M_{2}$ are congruent modulo $p$. There then exist a finite extension $k$ of $\Fp_{p}$, a finite set of primes $\Sigma$ containing $p$ and a determinant
\begin{equation}\nonumber
\Dbar:G_{\Q,\Sigma}\fleche k
\end{equation}
such that the determinant of the residual representations attached to $M_{1}$ and $M_{2}$ are both equal to $\Dbar$. By \cite[Proposition 3.3]{ChenevierDeterminant}, the functor sending a profinite local $W(k)$-algebra to the set of $W(k)$-valued determinants on $G_{\Q,S}$ deforming $\Dbar$ is then pro-representable. In particular, there exists a complete local noetherian $W(k)$-algebra $R$ and a determinant 
\begin{equation}\nonumber
D:R[G_{\Q,\Sigma}]\fleche R
\end{equation}
such that $D_{1}$ and $D_{2}$ are specializations of $D$. So $M_{1}$ and $M_{2}$ belong to the same $p$-adic family $\Fcali$. Hence, belonging to a $p$-adic family or being congruent are equivalent notions and we henceforth use freely one for the other.

More generally still, we might consider higher congruences between motives. Two rational motives $M_{1}$ and $M_{2}$ are congruent modulo $p^{n}$ if they belong to a common family parametrized by $R$ with maximal ideal $\mgot$ and if in addition their determinants $D_{1}$ and $D_{2}$ induce the same determinant 
\begin{equation}\nonumber
D:\Z/p^{n}\Z[G_{\Q}]\fleche \Z/p^{n}\Z.
\end{equation}
This notion is extended in the obvious way to the case of motives congruent modulo $\mgot^{n}$.
\subsubsection{Algebraic local Euler factors for families}
Let $S$ be a finite set of rational primes and let $R$ be a complete local noetherian domain with fraction field $\Kcal$. Let $\Fcali$ be a $p$-adic family parametrized by $R$ attached to a determinant $D$ of degree $n$ on $G_{\Q,S}$ as in the previous sub-section. There exists a unique semi-simple $G_{\Q,S}$-representation $(\Vcal,\rho,\bar{\Kcal})$ such that $D$ is equal to $\det\circ\rho$. Let $\ell\neq p$ be a finite prime. We denote by $D^{I_{\ell}}$ the determinant on $G_{\Q_{\ell}}/I_{\ell}$ with values \textit{a priori} in the normalization $\Rcal$ of $R$ equal to the determinant of $\rho|_{G_{\Q_{\ell}}}$ restricted to $\Vcal^{I_{\ell}}$ and we denote by $\Fcali^{I_{\ell}}$ the corresponding $p$-adic family. 

The assumption that $R$ is a domain can be relaxed in a variety of situation. We state the following strengthening as it is most useful to us. Let $R$ be a complete local reduced noetherian ring with algebraically closed residue field. Assume that the residual representation $\rhobar$ attached to $D$ is absolutely irreducible.  Then there exists by \cite[Theorem 2.22]{ChenevierDeterminant} a $G_{\Q,S}$-representation $(\Tcal,\rho,R)$ such that $D=\det\circ\rho$ and, as in the previous paragraph, we define $D^{I_{\ell}}$ to be the $\Rcal$-valued determinant on $G_{\Q_{\ell}}/I_{\ell}$ restricted to $\Tcal^{I_{\ell}}$.

In both cases, we say that $\Fcali$ is unramified at $\ell$ if $D^{I_{\ell}}$ is of degree $n$. We note the obvious but important fact that if $\psi$ is a specialization of $\Rcal$ with values in a field $F$, then $\psi\circ D^{I_{\ell}}$ has no reason to be equal to $(\psi\circ D)^{I_{\ell}}$.

\begin{DefEnglish}
The algebraic local polynomial determinant at $\ell$ is the projective (hence free) $\Rcal[X]$-module 
\begin{equation}\nonumber
\Lcali_{\ell}(\Fcali,X)=\Det_{\Rcal[X]}\left[\Rcal[X]\overset{\chi(D^{I_{\ell}},\Fr(\ell))}{\fleche}\Rcal[X]\right]
\end{equation}
where the two $\Rcal[X]$ are placed in degree 0 and 1 respectively. The algebraic local determinant at $\ell$ is the base-change
\begin{equation}\nonumber
\Lcali_{\ell}(\Fcali)=\Lcali_{\ell}(\Fcali,X)\tenseur_{\Rcal[X],\pi}\Rcal.
\end{equation}
under the quotient map
\begin{equation}\nonumber
\pi:\Rcal[X]\fleche\Rcal[X]/(X-1).
\end{equation}
The algebraic local Euler factor $\Eul_{\ell}(\Fcali,X)\in\Rcal[X]$ at $\ell$ is $\chi(D^{I_{\ell}},\Fr(\ell))$. 
\end{DefEnglish}
The constant coefficient of a characteristic polynomial is equal to 1 so $\chi(\Fr(\ell))$ is a regular element of $\Rcal[X]$ and so the complex 
\begin{equation}\nonumber
\left[\Rcal[X]\overset{\chi(D^{I_{\ell}},\Fr(\ell))}{\fleche}\Rcal[X]\right]
\end{equation}
is acyclic after tensor with the total ring of fractions $Q(\Rcal[X])$ of $\Rcal[X]$. The algebraic local Euler factor $\Eul_{\ell}(\Fcali,X)$ is thus also the monic generator of the ideal equal to the image of the algebraic local determinant
\begin{equation}\nonumber
\Det_{\Rcal[X]}\left[\Rcal[X]\overset{\chi(D^{I_{\ell}},\Fr(\ell))}{\fleche}\Rcal[X]\right]
\end{equation}
through the map 
\begin{align}\nonumber
\Det_{\Rcal[X]}\left[\Rcal[X]\overset{\chi(D^{I_{\ell}},\Fr(\ell))}{\fleche}\Rcal[X]\right]&\plonge\Det_{Q(\Rcal[X])}\left[Q(\Rcal[X])\overset{\chi(D^{I_{\ell}},\Fr(\ell))}{\fleche}Q(\Rcal[X])\right]\\\nonumber
&\isocan\Det_{Q(\Rcal[X])}(0)\isocan Q(\Rcal[X])
\end{align}
where the last canonical isomorphism is the canonical isomorphism required in the definition of the determinant functor.
\begin{LemEnglish}\label{LemParfait}
Assume that there exists a $G_{\Q_{\ell}}$-representation $(\Tcal,\rho,\Rcal)$ such that $D^{I_{\ell}}=\det\circ\rho^{I_{\ell}}$ and that $\Tcal^{I_{\ell}}$ is an $\Rcal$-module of finite projective dimension. Then there is a canonical isomorphism 
\begin{equation}\nonumber
\Lcali_{\ell}(\Fcali)\isocan\Det_{\Rcal}\RGamma(G_{\Q_{\ell}}/I_{\ell},\Tcal^{I_{\ell}})
\end{equation}
\end{LemEnglish}
\begin{proof}
The complex $\RGamma(G_{\Q_{\ell}}/I_{\ell},\Tcal^{I_{\ell}})$ is represented by the complex
\begin{equation}\nonumber
[\Tcal^{I_{\ell}}\overset{1-\Fr(\ell)}{\fleche}\Tcal^{I_{\ell}}]
\end{equation}
in degree 0 and 1. Writing a projective resolution of $\Tcal^{I_{\ell}}$ yields the result.
\end{proof}
By lemma \ref{LemParfait}, the algebraic local determinant is a generalization to the determinant of the complex of local cochains unramified at $\ell$. 
\begin{TheoEnglish}\label{PropLocal}
Let $\Fcali$ be a $p$-adic family parametrized by a complete local domain $R$ with normalization $\Rcal$. Suppose the set of pure specializations of $\Fcali$ is not empty and let $\psi$ be a pure specialization. Let $\phi:\Rcal\fleche\Scal$ be a ring morphism with value in a normal domain $\Scal$ such that $\ker\phi\subset\ker\psi$ and denote by $\Fcal_{\phi}$ the $p$-adic family attached to $D_{\phi}$. Then there exists a canonical isomorphism
\begin{equation}\nonumber
\Lcali_{\ell}(\Fcali_{\phi},X)\isocan\Lcali_{\ell}(\Fcali,X)\tenseur_{\Rcal,\phi}\Scal
\end{equation}
and thus
\begin{equation}\nonumber
\Eul_{\ell}(\Fcal_{\phi},X)=\phi\left(\Eul_{\ell}(\Fcal,X)\right).
\end{equation}
Equivalently, algebraic local Euler factors are compatible with pure specializations.
\end{TheoEnglish}
This is the main result of \cite[Chapter I]{TheseSaha}. We outline the proof.
\begin{proof}
Let $F$ be the fraction field of $\Scal$. The representations 
\begin{equation}\nonumber
\rho:G_{\Q_{\ell}}\fleche\GL_{n}(\bar{\Kcal}),\ \rho_{\phi}:G_{\Q_{\ell}}\fleche\GL_{n}(\Fbar)
\end{equation}
and 
\begin{equation}\nonumber
\rho_{\psi}:G_{\Q_{\ell}}\fleche\GL_{n}(\Qbar_{p})
\end{equation}
satisfy the hypotheses of Grothendieck's monodromy theorem. Thus, for instance, $\rho$ is unipotent after restriction to an open subgroup $U$ of $I_{\ell}$ with finite index and there exists a monodromy matrix $N$ such that for all $\tau\in U$, $\rho(\tau)=\exp(t(\tau)N)$ and verifying
\begin{equation}\label{EqMonoRel}
\rho(\s)N=\ell^{-1}N\rho(\s)
\end{equation}
if $\s$ is a lift of the Frobenius morphism. Here $t$ denotes the map from $I_{\ell}$ to its tame quotient and then to $\zp$ (strictly speaking, it depends on a choice of compatible system of primitive roots of unity). Similar statements are also true of $\rho_{\phi}$ and $\rho_{\psi}$ and we denote by $N_{\phi}$ and $N_{\psi}$ their monodromy matrices. 

The maximal possible rank $r$ of the monodromy matrix $N$ can be read off the powers of $\ell$ appearing in the quotients of the eigenvalues of a lift $\s$ of the Frobenius morphism thanks to the relation \eqref{EqMonoRel}. Combining this with \cite[Proposition 7.8.19]{BellaicheChenevier} shows the inequality
\begin{equation}\label{EqIneq}
\rank_{\Qbar_{p}}N_{\psi}\leq\rank_{\Fbar}N_{\phi}\leq\rank_{\bar{\Kcal}}N\leq r.
\end{equation}
As $\psi$ is a $\zp$-algebra morphism, any power of $\ell$ appearing in the quotients of the eigenvalues of $\rho(\s)$ also appears in the quotients of the eigenvalues of $\rho_{\psi}(\s)$. In \cite[Section 1.3]{TheseSaha}, it is shown that this implies that the weight filtration of $\rho_{\psi}$ has enough non-trivial graded pieces. Because $\rho_{\psi}$ is a pure representation, its monodromy filtration has the same non-trivial graded pieces. As in \cite[Proposition 1.3.4]{TheseSaha}, this implies that $N_{\psi}$ has at least rank $r$ and hence that all the inequalities in \eqref{EqIneq} are equalities.

In particular, the dimension of the $I_{\ell}$-invariants of $\rho$, $\rho_{\phi}$ and $\rho_{\psi}$ (over $\bar{\Kcal}$, $\Fbar$ and $\Qbar_{p}$ respectively) are equal and the base change properties of determinant implies that there is canonical isomorphism
\begin{equation}\nonumber
\Lcali_{\ell}(\Fcali_{\phi},X)\isocan\Lcali_{\ell}(\Fcali,X)\tenseur_{\Rcal,\phi}F
\end{equation}
and thus that
\begin{equation}\nonumber
\Eul_{\ell}(\Fcal_{\phi},X)=\phi\left(\Eul_{\ell}(\Fcal,X)\right).
\end{equation}
\end{proof}
We note that the results of \cite{TheseSaha} are in fact a good deal stronger, as they show that the full structure of Weil-Deligne representations is preserved through pure specializations.
\subsubsection{Algebraic $p$-adic $L$-functions for $p$-adic families}
Equipped with theorem \ref{PropLocal}, it is tempting to simply define the algebraic $p$-adic determinant of a family $\Fcali$ as the tensor product at all primes $\ell\neq p$ of the inverse of the algebraic local determinants and the algebraic $p$-adic $L$-function as the product of the inverse of local algebraic Euler factors just as the complex $L$-function is the product of local Euler factors. Though this entails a tensor product over an infinite number of primes and so could not possibly literally work, it almost does.

Let $\Fcali$ be a $p$-adic family parametrized by a complete local reduced noetherian ring $R$ and attached to $D$. Assume that $\Fcali$ is unramified outside a finite set $\Sigma$ of finite primes. Henceforth, we also make the following assumption
\begin{HypEnglish}\label{HypIrr}
Assume that the residual representation of $\Fcali$ is absolutely irreducible.
\end{HypEnglish}
Hence there exists a $G_{\Q,\Sigma}$-representation $(\Tcal,\rho,R)$ such that $D=\det\circ\rho$.
\begin{DefEnglish}\label{DefLalg}
Let $\Fcali$ be a $p$-adic family satisfying assumption \ref{HypIrr} parametrized by a complete local reduced noetherian ring $R$ with normalization $\Rcal$. Let $S\subset\Sigma$ be a finite set of rational primes containing $p$. The $S$-partial algebraic $p$-adic determinant of $\Fcali$ is the free $\Rcal$-module of rank 1
\begin{equation}\nonumber
\Lcali_{S,\Sigma}(\Fcali)=\Det^{-1}_{\Rcal}\RGamma_{c}(\Z[1/\Sigma],\Tcal\tenseur\Rcal)\tenseur_{\Rcal}\produittenseur{\ell\in\Sigma\backslash S}{}\Lcali^{-1}_{\ell}(\Fcali).
\end{equation}
\end{DefEnglish}
The $S$-partial algebraic $p$-adic $L$-function depends in full generality on both the choices of $\Sigma$ and $S$. We record the following lemma which guarantees the independence on $\Sigma$ of $\Lcali_{\{p\},\Sigma}$ and justifies our omission of $\Sigma$ when referring to $\Lcali_{\{p\},\Sigma}(\Fcali)$.
\begin{LemEnglish}\label{LemIndep}
Let $\Fcali$ be a $p$-adic family as in definition \ref{DefLalg}. Assume that $\Fcali$ is unramified outside a finite set $\Sigma_{0}$ of finite primes containing $p$. Then $\Lcali_{\{p\}}(\Fcali)$ does not depend on the choice of $\Sigma\supset\Sigma_{0}$ in the sense that there exists a canonical isomorphism between
\begin{equation}\nonumber
\Lcali_{\{p\},\Sigma}(\Fcali)=\Det^{-1}_{\Rcal}\RGamma_{c}(\Z[1/\Sigma],\Tcal\tenseur\Rcal)\tenseur_{\Rcal}\produittenseur{\ell\in\Sigma\backslash\{p\}}{}\Lcali^{-1}_{\ell}(\Fcali)
\end{equation}
and 
\begin{equation}\nonumber
\Lcali_{\{p\},\Sigma'}(\Fcali)=\Det^{-1}_{\Rcal}\RGamma_{c}(\Z[1/\Sigma'],\Tcal\tenseur\Rcal)\tenseur_{\Rcal}\produittenseur{\ell\in\Sigma'\backslash\{p\}}{}\Lcali^{-1}_{\ell}(\Fcali)
\end{equation}
for all $\Sigma,\Sigma'$ containing $\Sigma_{0}$.
\end{LemEnglish}
\begin{proof}
Without loss of generality, we can assume that $\Sigma\subset\Sigma'$. There is a canonical isomorphism between $\Det^{-1}_{\Rcal}\RGamma_{c}(\Z[1/\Sigma],\Tcal\tenseur\Rcal)$ and
\begin{equation}\nonumber
\produittenseur{\ell\in\Sigma\backslash\{p\}}{}\Det_{\Rcal}\RGamma(G_{\Q_{\ell}},\Tcal\tenseur\Rcal)\tenseur_{\Rcal}{\Det^{-1}_{\Rcal}\RGamma(G_{\Q,\Sigma},\Tcal\tenseur\Rcal)}.
\end{equation}
Because a cochain unramified outside $\Sigma$ is a cochain unramified outside $\Sigma'$ which is further unramified at the places in $\Sigma'\backslash\Sigma$ and because $\Tcal\tenseur\Rcal$ is unramified at $\ell\in\Sigma'\backslash\Sigma$, this latter projective module is canonically isomorphic to
\begin{equation}\nonumber
\produittenseur{\ell\in\Sigma'\backslash\{p\}}{}\Det_{\Rcal}\RGamma(G_{\Q_{\ell}},\Tcal\tenseur\Rcal)\tenseur_{\Rcal}{\Det^{-1}_{\Rcal}\RGamma(G_{\Q,\Sigma'},\Tcal\tenseur\Rcal)}\produittenseur{\ell\in\Sigma'\backslash\Sigma}{}\Det^{-1}_{\Rcal}\RGamma(G_{\Q_{\ell}},\Tcal\tenseur\Rcal).
\end{equation}
and thus to
\begin{equation}\nonumber
\Det^{-1}_{\Rcal}\RGamma_{c}(\Z[1/\Sigma'],\Tcal\tenseur\Rcal)\tenseur\produittenseur{\ell\in\Sigma'\backslash\Sigma}{}\Det^{-1}_{\Rcal}\RGamma(G_{\Q_{\ell}},\Tcal\tenseur\Rcal).
\end{equation}
To prove the lemma, it is consequently enough to show that $\Lcali_{\ell}(\Fcali)$ is canonically isomorphic to $\Det_{\Rcal}\RGamma(G_{\Q_{\ell}},\Tcal\tenseur\Rcal)$ when $\ell$ belongs to $\Sigma'\backslash\Sigma$. Since $\Tcal\tenseur\Rcal$ is unramified at such $\ell$, the module $(\Tcal\tenseur\Rcal)^{I_{\ell}}$ has finite projective dimension over $\Rcal$ and the last statement is then true by lemma \ref{LemParfait}.
\end{proof}
By construction, algebraic $p$-adic determinants with coefficients in normal domains commute with pure specializations in the sense of the following proposition.
\begin{Prop}\label{PropSpec}
Let $\Fcali$ be a $p$-adic family as in definition \ref{DefLalg} parametrized by a complete local normal noetherian domain $\Rcal$. Let $\lambda:\Rcal\fleche\Scal$ be a $\zp$-algebra morphisms with values in a normal domain $\Scal$ such that $\ker\lambda$ is included in the kernel of a motivic specialization. There then exists a canonical isomorphism 
\begin{equation}\nonumber
\Lcali_{S}(\Fcali)\tenseur_{\Rcal,\lambda}\Scal\isocan\Lcali_{S}(\Fcali_{\lambda}).
\end{equation}
\begin{proof}
Compactly supported étale cohomology complexes commute with base-change of ring of coefficients so
\begin{equation}\nonumber
\RGamma_{c}(\Z[1/\Sigma],\Tcal)\Ltenseur_{\Rcal,\lambda}\Scal=\RGamma_{c}(\Z[1/\Sigma],\Tcal\tenseur_{\Rcal,\lambda}\Scal).
\end{equation}
It is thus enough to show that $\Lcali_{\ell}(\Fcali)\tenseur_{\Rcal,\lambda}\Scal$ is canonically isomorphic to $\Lcali_{\ell}(\Fcali_{\lambda})$ for all $\ell\in\Sigma$ different from $p$. As $\ker\lambda$ is included in the kernel of a motivic specialization and as motivic specializations are pure in the sense of assumption \ref{HypWMC} at all primes $\ell$ by assumption, this last statement is implied by theorem \ref{PropLocal}.
\end{proof}
\end{Prop}
\paragraph{Cyclotomic deformations}\label{SubCyc} In this paragraph, we justify definition \ref{DefLalg} by computing $\Lcali_{S}(\Fcali)$ in the simplest non-trivial example of a $p$-adic family: the $p$-adic family of cyclotomic twists of a motivic $G_{\Q}$-representation. Let $(T,\rho,\zp)$ be a $G_{\Q}$-representation appearing as a $\zp$-lattice in the $p$-adic étale realization of a pure motive over $\Q$. Recall from subsection \ref{SubComp} that $\Lambda$ is the completed group algebra $\zp[[\Gal(\Q_{\infty}/\Q)]]$. The cyclotomic deformation $(T\tenseur\Lambda,\rho_{\Lambda},\Lambda)$ is the $G_{\Q}$-representation $T\tenseur_{\zp}\Lambda$ with $G_{\Q}$-action on both sides of the tensor product. It may arise in two not completely formally equivalent ways: first, as limit of compatible ETNC over finite abelian $p$-extensions; second, as the $p$-adic family constructing by taking the universal deformation of the determinant of $T$ (the equivalence following easily by class field theory). Because $\Lambda$ is a regular local ring, the Auslander-Buchsbaum and Serre's theorem guarantees that all bounded complexes of $\Lambda$-modules are perfect and so the determinants of all bounded complexes appearing below are well-defined. Similarly, $\Lambda$ is normal so that local algebraic Euler factors and determinants have coefficients in $\Lambda$.

By Shapiro's lemma, the étale cohomology complex $\RGamma_{\et}(\Z[1/S],T\tenseur\Lambda)$ is also equal to the inverse limit
\begin{equation}\nonumber
\limproj{n}\RGamma_{\et}(\Ocal_{n}[1/S],T)
\end{equation}
and to the inverse limit
\begin{equation}\nonumber
\limproj{n}\ e\RGamma_{\et}(\Z[\zeta_{p^{n}},1/S],T)
\end{equation}
where $e$ is the projection from $\Gal(\Q(\zeta_{p^{\infty}})/\Q)\simeq\Gamma\times(\Z/(p-1)\Z)$ onto $\Gamma$ which is trivial on $\Z/(p-1)\Z$.

Let $\Sigma$ be a finite set of primes containing $p$ and the primes of ramification of $T$. For $\ell\neq p$, the $G_{\Q_{\ell}}$-representation $\Lambda$ is unramified and so there are canonical isomorphisms
\begin{equation}\nonumber
\Lcali_{\ell}(T\tenseur\Lambda)\isocan\Det_{\Lambda}\RGamma(G_{\Q_{\ell}}/I_{\ell},(T\tenseur\Lambda)^{I_{\ell}})\isocan\Det_{\Lambda}\RGamma(G_{\Q_{\ell}}/I_{\ell},T^{I_{\ell}}\tenseur\Lambda)
\end{equation}
by lemma \ref{LemParfait}. These isomorphisms induce together a canonical isomorphism
\begin{equation}\nonumber
\Lcali_{\{p\}}(T\tenseur\Lambda)\isocan\Det^{-1}_{\Lambda}\RGamma_{\et}(\Z[1/p],T\tenseur\Lambda).
\end{equation}
The algebraic $p$-adic determinant of $T\tenseur\Lambda$ is thus canonically isomorphic to the inverse of the determinant of the \Nekovar-Selmer complex of $T\tenseur\Lambda$.

In order to relate this determinant to a genuine element of $\Lambda$, we assume further for simplicity that the residual representation $\rhobar$ of $T$ is absolutely irreducible. Then $H^{0}_{\et}(\Z[1/p],T\tenseur\Lambda)$ vanishes and so does $H^{3}_{\et}(\Z[1/p],T\tenseur\Lambda)$ by Artin-Verdier duality (reducing here to Poitou-Tate duality). The spectral sequence attached to the functors $\Gamma(\Z[1/p],-)$ and $-\tenseur_{\Lambda}\Lambda/x$ shows that $H^{1}_{\et}(\Z[1/p],T\tenseur\Lambda)$ is a $\Lambda$-module of depth 2, and hence a free module by the Auslander-Buchsbaum formula. Assume further that the weak Leopoldt's conjecture of \cite[Conjecture Section 1.3]{PerrinRiouLpadique} is true for $T$, that is to say that $H^{2}(\Z[1/p],T\tenseur\Lambda)$ is a torsion $\Lambda$-module. The injection
\begin{equation}\nonumber
\Det^{-1}_{\Lambda}\RGamma_{\et}(\Z[1/p],T\tenseur\Lambda)\plonge\Det^{-1}_{\Frac{\Lambda}}\RGamma_{\et}(\Z[1/p],T\tenseur\Frac(\Lambda))
\end{equation}
followed by the canonical isomorphisms 
\begin{equation}\nonumber
\Det^{-1}_{\Frac{\Lambda}}\RGamma_{\et}(\Z[1/p],T\tenseur\Frac(\Lambda))\isocan\Det^{-1}_{\Frac(\Lambda)}(0)\isocan\Frac(\Lambda)
\end{equation}
induced by acyclicity of $\RGamma_{\et}(\Z[1/p],T\tenseur\Frac(\Lambda))$ and the functorial properties of the determinant functor then sends $\Det^{-1}_{\Lambda}\RGamma_{\et}(\Z[1/p],T\tenseur\Lambda)$ to the free $\Lambda$-module
\begin{equation}\nonumber
\frac{\Det_{\Lambda}\Hun_{\et}(\Z[1/p],T\tenseur\Lambda)}{\Det_{\Lambda}H^{2}_{\et}(\Z[1/p],T\tenseur\Lambda)}=\Det_{\Lambda}\Hun_{\et}(\Z[1/p],T\tenseur\Lambda)\carac_{\Lambda}H^{2}_{\et}(\Z[1/p],T\tenseur\Lambda)
\end{equation}
in $\Frac(\Lambda)$ (the last equality by localizing at all grade 1 primes of $\Lambda$ and by the structure theorem for modules over discrete valuation rings). Choose a basis of the free $\Lambda$-module $\Hun_{\et}(\Z[1/p],T\tenseur\Lambda)$. This choice yields a trivialization of $\Det_{\Lambda}\Hun_{\et}(\Z[1/p],T\tenseur\Lambda)$ which, when composed with all the maps above, sends
$\Lcali_{\{p\}}(T\tenseur\Lambda)$ to $\carac_{\Lambda}H^{2}_{\et}(\Z[1/p],T\tenseur\Lambda)$ and so to the characteristic ideal of the Selmer group appearing as algebraic $p$-adic $L$-function in Kato's formulation of the Iwasawa Main Conjecture for motives.

Note further that the chosen basis of $\Hun_{\et}(\Z[1/p],T\tenseur\Lambda)$ produces a basis of $\Hun_{\et}(\Z[1/p],T\tenseur\Lambda/x)$ for all $x$ such that $H^{2}_{\et}(\Z[1/p],T\tenseur\Lambda)[x]=0$ via the short exact sequence
\begin{equation}\nonumber
\suiteexacte{}{}{\Hun_{\et}(\Z[1/p],T\tenseur\Lambda)/x}{\Hun_{\et}(\Z[1/p],T\tenseur\Lambda/x)}{H^{2}_{\et}(\Z[1/p],T\tenseur\Lambda)[x]}.
\end{equation}
At least for these $x$ then, so for all $x$ except possibly finitely many and conjecturally for all $x$ such that the $L$-function of $T\tenseur\Lambda/x$ does not vanish at zero, the choice of a basis of $\Hun_{\et}(\Z[1/p],T\tenseur\Lambda)$ yields a coherent choice of trivialization of $\Det_{\Lambda/x}\Hun_{\et}(\Z[1/p],T\tenseur\Lambda/x)$.

\paragraph{General expectation} Based on the example of cyclotomic deformations and taking into account the base-change properties of algebraic determinants at motivic points, we record here the guiding (but at present necessarily vague) principle underlying this manuscript.
\begin{Principle}\label{Principle}
The data of the algebraic special values of a $p$-adic family of motives is equivalent to the data of the algebraic $p$-adic determinant $\Lcali(\Fcali)$ of $\Fcali$ plus a choice of a compatible system of trivializations.
\end{Principle}
In the case of cyclotomic deformation, this choice of a compatible system of trivializations can be somewhat artificially made using the peculiar commutative algebra properties of $\Lambda$ as above, but it might be difficult to define for general motives. For $p$-adic families of automorphic motives, however, there is a reasonable conjectural choice that we describe below: the trivialization coming from the choice of a morphism from the completed cohomology of \cite{EmertonInterpolationEigenvalues} to étale cohomology.

\subsubsection{Algebraic determinants of automorphic families and the Local Langlands Correspondence}
Until now, we have worked with abstract $p$-adic families. This has had the unfortunate consequence to require a potential base change of rings of coefficients to the normalization of the coefficient ring $R$ in the very definition of the algebraic determinants. In actual known cases, however, the construction of the family $\Fcali$ proceeds either through twists in $\zp^{d}$-extension (plus possibly twists by a finite group), in which case the construction proceeds as an inverse limit of maps $\Spec\Ocal_{K}[1/p]\fleche\Spec\Z[1/p]$ and everything is well-defined without passing to the normalization by proper base change, or through $p$-adic interpolation of automorphic representations so that the determinant $D^{I_{\ell}}$ recovers the local automorphic Euler factor in a dense set of points and thus has values in $R$ itself. We record these facts as a theorem.

Let $\Sigma$ be a finite set of finite places containing $p$. Let $\G$ be a reductive group over $\Q$ and let $K_{\Sigma}$ be a compact open subgroup of $\G(\A^{(p\infty)}_{\Q})$ which is maximal hyperspecial for $\ell\notin\Sigma$. The abstract $p$-adic Hecke algebra
\begin{equation}\nonumber
\Hecke_{\Sigma}=\zp[K_{\Sigma}\backslash\G(A_{\Q}^{(p\infty)})/K_{\Sigma}]
\end{equation} is the algebra of left and right $K_{\Sigma}$-invariant compactly supported functions on $\G(\A_{\Q}^{(p\infty)})$.
\begin{TheoEnglish}\label{TheoAuto}
Let $R_{\Sigma}$ be a local factor of the reduced $p$-adic Hecke algebra generated by operators outside $\Sigma$ or a quotient thereof by a minimal prime ideal. Let $\Fcali$ be a $p$-adic family parametrized by $R_{\Sigma}$ as in definition \ref{DefLalg} (in particular the residual representation of $\Fcali$ is absolutely irreducible). Assume that there exists a Zariski-dense subset $\Hom^{\cl}(R_{\Sigma},\Qbar_{p})$ of $\Hom(R_{\Sigma},\Qbar_{p})$ such that for all $\psi\in\Hom^{\cl}(R,\Qbar_{p})$, $\Fcali_{\psi}$ is a pure $G_{\Q,\Sigma}$-representation attached to an automorphic representation $\pi({\psi})$ of $\G(\A_{\Q}^{(p\infty)})$ satisfying Langlands local-global compatibility conjecture. 

Then, for all $\ell\neq p$, $\Lcali_{\ell}(\Fcali,X)$ is a well-defined invertible $R_{\Sigma}[X]$-module, $\Eul_{\ell}(\Fcali,X)$ is a well-defined element of $R_{\Sigma}[X]$ and $\Lcali_{S}(\Fcali)$ is a well-defined invertible $R_{\Sigma}$-module. Theorem \ref{PropLocal}, proposition \ref{PropSpec} and lemmas \ref{LemParfait} and \ref{LemIndep} remain true with the normal ring $\Rcal$ replaced by the ring $R_{\Sigma}$ and all other hypotheses unchanged (most especially, theorem \ref{PropLocal} and proposition \ref{PropSpec} are conditional on $R_{\Sigma}$ being a domain).
\end{TheoEnglish}
\begin{proof}
Once it is shown that $D^{I_{\ell}}$ has values in $R_{\Sigma}$, all the other assertions follow with the same proof. Let $\psi$ be a specialization in $\Hom^{\cl}(R_{\Sigma},\Qbar_{p})$. Then the determinant $(D^{I_{\ell}})_{\psi}$ is equal to $D_{\psi}^{I_{\ell}}$ by theorem \ref{PropLocal}. By the assumptions that $D_{\psi}$ comes from an automorphic representation $\pi(\psi)$ satisfying the local-global compatibility so $D_{\psi}^{I_{\ell}}(1-X\Fr(\ell))$ is the automorphic $L$-factor of $\pi(\psi)$ and thus has values in the abstract Hecke algebra (acting on the space of automorphic forms of same level as $\pi(\psi)$).

The result then follows from the general fact that collection of determinants 
\begin{equation}\nonumber
\{D_{\psi}:G_{\Q,\Sigma}\fleche\Qbar_{p}|\psi\in\Hom^{\cl}(R_{\Sigma},\Qbar_{p})\}
\end{equation}
on a Zariski-dense subset $Z=\Hom^{\cl}(R_{\Sigma},\Qbar_{p})$ of $\Hom(R_{\Sigma},\Qbar_{p})$ can be interpolated uniquely to a determinant on the whole space. This is is presumably well-known but as we don't know a reference, we give a proof paraphrasing the parallel result for pseudo-characters in \cite{ChenevierFamilles}. The ring morphism 
\begin{equation}\nonumber
\phi:R_{\Sigma}\fleche\produit{\psi\in Z}{}\Qbar_{p}
\end{equation}
sending $x$ to $(\psi(x))_{\psi}$ is injective with closed image and induces an homeomorphism from $R_{\Sigma}$ to its image. The map 
\application{\pi}{G_{\Q,\Sigma}}{\produit{\psi\in Z}{}\Qbar_{p}}{g}{\produit{\psi\in Z}{}D_{\psi}(g)=\produit{\psi\in Z}{}\det\circ\rho_{\psi}(g)}
is a determinant with values in the image of $\phi$. The determinant
\begin{equation}\nonumber
D:R_{\Sigma}[G_{\Q,\Sigma}]\fleche R_{\Sigma}
\end{equation}
is then the composition $\phi^{-1}\circ\pi$ and is the unique continuous map coinciding with $D_{\psi}$ on $\psi\in Z$. Applying this general construction to the collection of the $D^{I_{\ell}}_{\psi}$ shows that $D^{I_{\ell}}$ has values in $R_{\Sigma}$.
\end{proof}
As an application, we can construct an algebraic $p$-adic determinant for ordinary families of automorphic representations of definite unitary groups.
\begin{TheoEnglish}\label{TheoUnitary}
Let $\G$ be a definite unitary groups over an imaginary quadratic extension of $\Q$. Let $R_{\Sigma}$ be a local factor of the nearly-ordinary reduced Hecke algebra generated by operators outside $\Sigma$ attached to a residually absolutely irreducible representation. Let $\aid$ be a minimal prime of $R_{\Sigma}$. Then there exists a $p$-adic family $\Fcali$ of $G_{\Q,\Sigma}$-representations with coefficients in $R_{\Sigma}/\aid$ and an algebraic determinant $\Lcali_{\Sigma}(\Fcali)$ satisfying the interpolation property
\begin{equation}\nonumber
\Lcali_{\{p\}}(\Fcali)\tenseur_{R_{\Sigma}/\aid,\psi}\Ocal\isocan\Lcali_{\{p\}}(\Fcali\tenseur_{R_{\Sigma}/\aid,\psi}\Ocal)
\end{equation}
for all $\psi$ attached to classical automorphic representations of $\G(\A_{\Q}^{(\infty)})$.
\end{TheoEnglish}
We refer to \cite[Chapter 4]{TheseSaha} and especially Theorem 4.3.6 therein for precise statements and proof.

Because this is our case of interest and because we don't know any situation where this hypothesis is not satisfied, we henceforth make without further comment the hypothesis that the determinant $D^{I_{\ell}}$ has values in $R$ itself instead of its normalization and therefore that the algebraic $p$-adic determinant is an $R$-module.

\paragraph{Determinants without representations}We also record here the following interesting question. The $L$-function of a motive is defined independently of the choice of a lattice in $M_{\et,p}$ and indeed it is a major feature of conjectures \ref{ConjTNC} and \ref{ConjETNC} that they are independent of such a choice. Is it possible to define the algebraic $p$-adic determinant of a determinant $D$ (coming from $p$-adic interpolation of automorphic representations) without assumption \ref{HypIrr}? More precisely, is it possible to attach to a determinant $D$ with values in $R$ a functor $F_{D}$ from the category of noetherian $\zp$-algebra with residue field equal to $R/\mgot$ to the category of sets sending a ring $S$ to the singleton formed of a graded invertible $S$-module with the property that $F_{D}(S)=\Det_{R}\RGamma_{c}(\Z[1/\Sigma],T)$ whenever there exists a $G_{\Q,\Sigma}$-representation $(T,\rho,S)$ such that $D\tenseur_{R}S=\det\circ\rho$ and satisfying the usual compatibility properties of the determinant functor? A positive answer to this question would allow the construction of algebraic $p$-adic determinants attached to $p$-adic families of automorphic motives.

\subsection{Mazur's question}
\subsubsection{A local obstruction}\label{SubObstr}
\newcommand{\triv}{\operatorname{triv}}
The very construction of the algebraic $p$-adic determinant of a $p$-adic family of motives $\Fcali$ provides an immediate local obstruction for two congruent motives to have the same algebraic determinant modulo $p$ and thus for Mazur's question to admit a positive answer in its most literal form.

Indeed, let $\psi$ and $\phi$ be two motivic specializations which we view as both having values in a sufficiently large discrete valuation ring $\Ocal$ flat over $\zp$. We fix a finite set $\Sigma$ such that $M_{\psi}$ and $M_{\phi}$ are unramified outside $\Sigma$. Let $\varpi$ be a uniformizing parameter of $\Ocal$ and let $k$ be the residual field $\Ocal/\varpi$. In order to compare the algebraic $p$-adic determinants $\Lcali_{\{p\}}(\Fcali_{\psi})$ and $\Lcali_{\{p\}}(\Fcali_{\phi})$ modulo $\varpi$, we first note that
\begin{align}\nonumber
\left(\Lcali_{\{p\}}(\Fcali_{\psi})\tenseur\produittenseur{\ell\in\Sigma\backslash\{p\}}{}\Lcali_{\ell}(\Fcali_{\psi})\right)\tenseur_{\Ocal}k&=\Det^{-1}_{k}\RGamma_{c}(\Z[1/\Sigma],T_{\psi})\tenseur_{\Ocal}k\\\nonumber
&=\Det^{-1}_{k}\RGamma_{c}(\Z[1/\Sigma],T_{\psi}\tenseur_{\Ocal}k)\\\nonumber
&\isocan\Det^{-1}_{k}\RGamma_{c}(\Z[1/\Sigma],T_{\phi}\tenseur_{\Ocal}k)\\\nonumber
&=\left(\Lcali_{\{p\}}(\Fcali_{\phi})\tenseur\produittenseur{\ell\in\Sigma\backslash\{p\}}{}\Lcali_{\ell}(\Fcali_{\phi})\right)\tenseur_{\Ocal}k.
\end{align}
Hence, we see that thanks to the base-change property of the compactly supported étale cohomology complex, there is a canonical isomorphism between the algebraic determinants of $T_{\psi}$ and $T_{\phi}$ modulo $\varpi$ with Euler factors at primes of ramification removed.
\begin{equation}\label{EqCanEuler}
\left(\Lcali_{\{p\}}(\Fcali_{\psi})\tenseur\produittenseur{\ell\in\Sigma\backslash\{p\}}{}\Lcali_{\ell}(\Fcali_{\psi})\right)\tenseur_{\Ocal}k\isocan\left(\Lcali_{\{p\}}(\Fcali_{\phi})\tenseur\produittenseur{\ell\in\Sigma\backslash\{p\}}{}\Lcali_{\ell}(\Fcali_{\phi})\right)\tenseur_{\Ocal}k
\end{equation}

As we have recorded under the form of principle \ref{Principle}, algebraic special values of motives correspond to algebraic determinants plus choices of trivializations so if the algebraic special values of $M_{\psi}$ and $M_{\psi}$ are congruent, we may assume that they are given by two trivializations $\triv_{\psi}$ and $\triv_{\phi}$ of $\Lcali_{\{p\}}(\Fcali_{\psi})$ and $\Lcali_{\{p\}}(\Fcali_{\phi})$ respectively coming from a common trivialization $\triv$ of $\Lcali_{\{p\}}(\Fcali)$ in the sense that the diagram
\begin{equation}\label{DiagTrivUn}
\xymatrix{
&\Lcali_{\{p\}}(\Fcali_{\psi})\ar[r]^(0.65){\triv_{\psi}}&\Ocal\ar[rd]^{\modulo\varpi}&\\
\Lcali_{\{p\}}(\Fcali)\ar^{-\tenseur_{R,\psi}\Ocal}[ru]\ar_{-\tenseur_{R,\phi}\Ocal}[rd]\ar[r]^(0.65){\triv}&R\ar[ru]_{\psi}\ar[rd]^{\phi}\ar[rr]^(0.6){\modulo\mgot}&&k\\
&\Lcali_{\{p\}}(\Fcali_{\phi})\ar[r]_(0.65){\triv_{\phi}}&\Ocal\ar[ru]_{\modulo\varpi}&
}
\end{equation}
commutes. This implies in particular that $\triv_{\psi}(\Lcali_{\{p\}}(\Fcali_{\psi}))$ and $\triv_{\phi}(\Lcali_{\{p\}}(\Fcali_{\phi}))$ are both units or both non-units in $\Ocal$ (note that for this particular consequence to be true, it is not necessary to know that special values can be $p$-adically interpolated to the whole of $\Fcali$, or more precisely it is not necessary that there be a trivialization map with values in $R$; the rightmost part of diagram \eqref{DiagTrivUn} is enough). 

Let us assume further for simplicity that all the eigenvalues of all lifts of the Frobenius morphisms at primes in $\Sigma\backslash\{p\}$ acting on $M_{\psi}$ and $M_{\phi}$ are different from 1. Under the permanent assumption \ref{HypWMC} that motivic Galois representations are pure representations, this can be ensured for instance by requiring that $M_{\psi}$ and $M_{\phi}$ are of odd weight. Then the maps $\triv_{\psi}$ and $\triv_{\phi}$ can be completed in trivializations $\triv^{*}_{\psi}$ and $\triv^{*}_{\phi}$ of 
\begin{equation}\nonumber
\Lcali_{\{p\}}(\Fcali_{\psi})\tenseur\produittenseur{\ell\in\Sigma\backslash\{p\}}{}\Lcali_{\ell}(\Fcali_{\psi})\textrm{ and }\Lcali_{\{p\}}(\Fcali_{\phi})\tenseur\produittenseur{\ell\in\Sigma\backslash\{p\}}{}\Lcali_{\ell}(\Fcali_{\phi})
\end{equation}
respectively by trivializing $\Lcali_{\ell}(\Fcali_{\pi})$ (for $\pi=\psi$ or $\phi$) for $\ell\in\Sigma\backslash\{p\}$ by taking tensors product with $\Frac(\Ocal)$ and noticing that the complex $[T_{\pi}^{I_{\ell}}\overset{1-\Fr(\ell)}{\fleche} T_{\pi}^{I_{\ell}}]\tenseur_{\Ocal}\Frac(\Ocal)$ is acyclic and thus has determinant canonically isomorphic to $\Frac(\Ocal)$. If 
\begin{equation}\nonumber
L_{p}(T_{\pi})=\triv_{\pi}(\Lcali_{\{p\}}(\Fcali_{\pi}))
\end{equation}
then
\begin{equation}\nonumber
\triv^{*}_{\pi}\left(\Lcali_{\{p\}}(\Fcali_{\pi})\tenseur\produittenseur{\ell\in\Sigma\backslash\{p\}}{}\Lcali_{\ell}(\Fcali_{\pi})\right)=L_{p}(T_{\pi})\produit{\ell\in\Sigma\backslash\{p\}}{}\Eul_{\ell}(\Fcali_{\pi},1).
\end{equation}
This construction is the algebraic equivalent of removing the Euler factors at primes of bad reduction familiar on the analytic side. The compatibility of the trivializations $\triv_{\psi}$ and $\triv_{\pi}$ expressed in the diagram \eqref{DiagTrivUn} and the canonical isomorphism \eqref{EqCanEuler} then implies that 
\begin{equation}\label{EqTriv}
\produit{\ell\in\Sigma\backslash\{p\}}{}\Eul_{\ell}(\Fcali_{\psi},1)\equiv\produit{\ell\in\Sigma\backslash\{p\}}{}\Eul_{\ell}(\Fcali_{\phi},1)\modulo\varpi
\end{equation}
are both units or non-units in $\Ocal$.

However, suppose now that there is a prime $\ell\in\Sigma\backslash\{p\}$ such that $\Eul_{\ell}(\Fcali_{\psi},1)$ is a unit in $\Ocal$ but such that $\Eul_{\ell}(\Fcali_{\phi},1)$ is not a unit and that for all other primes in $\ell\in\Sigma\backslash\{p\}$, the Euler factors of both $\Fcali_{\phi}$ and $\Fcali_{\psi}$ are units (or, more radically, assume that $\Sigma=\{\ell,p\}$). In that case, the right-hand term of \eqref{EqTriv} is a non-unit because of the Euler factor at $\ell$ and thus the same is true for the left-hand side, though crucially not because of the Euler factor at $\ell$. Hence, $L_{p}(T_{\psi})=\triv(\Lcali_{\{p\}}(\Fcali_{\psi}))$ has to be a non-unit and so has to be $L_{p}(T_{\phi})=\triv(\Lcali_{\{p\}}(\Fcali_{\phi}))$ if the family $\Fcali$ admits an algebraic $p$-adic $L$-function in the sense of diagram \eqref{DiagTriv}. Conversely, if  $L_{p}(T_{\phi})$ is a unit, then the family $\Fcali$ cannot admit an algebraic $p$-adic $L$-function nor can question \ref{Q2} of the introduction admit a positive answer.

\subsubsection{The positive case}
We see that the obstruction of sub-section \ref{SubObstr} may only occur for a $p$-adic family $\Fcali$ attached to a $G_{\Q,\Sigma}$-representation $(\Tcal,\rho,R)$ when one local algebraic Euler factor of $\Fcali_{\psi}$ is a unit but the same local algebraic Euler factor of $\Fcali_{\phi}$ is not a unit. According to theorem \ref{PropLocal}, this can only happen if $R$ is not a domain. 

In the converse direction, we have the following positive result. Assume for simplicity that 1 is an eigenvalue of $\Fr(\ell)$ acting on $\Tcal$ for no $\ell\in\Sigma$ (this is a relatively harmless assumption as any $p$-adic family can be twisted with the Iwasawa algebra $\Lambda$ by deforming the determinant of $\Tcal$ so that with a mild loss of generality, $R$ can be assumed to be a $\Lambda$-algebra; in which case $\Fcali$ contains motivic points of arbitrarily large weights whereas monodromy is constant in the family and so 1 is not an eigenvalue of $\Fr(\ell)$). Assume also that $R$ is a domain in order to rule out the obstruction of \ref{SubObstr}.

sLet $\ell$ be a prime in $\Sigma\backslash\{p\}$. Then the complex
\begin{equation}\nonumber
\left([\Tcal^{I_{\ell}}\overset{1-\Fr(\ell)}{\fleche}\Tcal^{I_{\ell}}]\right)\tenseur_{R}\Frac(R)
\end{equation} 
is acyclic so there is a canonical identification
\begin{equation}\nonumber
\Det_{\Frac(R)}\left([\Tcal^{I_{\ell}}\overset{1-\Fr(\ell)}{\fleche}\Tcal^{I_{\ell}}]\tenseur_{R}\Frac(R)\right)\isocan\Frac(R)
\end{equation}
of its determinant with $\Frac(R)$. Inside $\Frac(R)$ sits the canonical $R$-module
\begin{equation}\nonumber
\Det_{R}(0)\subset\Det_{\Frac(R)}(0)\isocan\Frac(R)
\end{equation}
which we identify with $R$. This construction yields a trivialization
\begin{equation}\nonumber
\triv_{\Fcali,\ell}:\Lcali_{\ell}(\Fcali)\plonge\Lcali_{\ell}(\Fcali)\tenseur_{R}\Frac(R)\isocan\Frac(R)
\end{equation}
of $\Lcali_{\ell}(\Fcali)$ sending a basis of $\Lcali_{\ell}(\Fcali)$ to $\Eul_{\ell}(\Fcali,1)$. Hence, any trivialization
\begin{equation}\nonumber
\triv_{c}:\Det^{-1}_{R}\RGamma_{c}(\Z[1/\Sigma],\Tcal)\fleche R
\end{equation}
of $\Det^{-1}_{R}\RGamma_{c}(\Z[1/\Sigma],\Tcal)$ can be extended to a trivialization $\triv_{\Fcali}$ of $\Lcali(\Fcali)$ by trivializing each $\Lcali_{\ell}(\Fcali)$ for $\ell\in\Sigma\backslash\{p\}$ by $\triv_{\Fcali,\ell}$ as above. The trivialization $\triv_{\Fcali}$ of $\Lcali(\Fcali)$ commutes with pure specializations $\psi$ such that 1 is not an eigenvalue of $\Fr(\ell)$ acting on $M_{\psi}$ in the following sense. First, for all specializations $\psi$ with values in $\Ocal_{\psi}$, $\triv_{c}$ induces a trivialization of $\Det^{-1}_{\Ocal_{\psi}}\RGamma_{c}(\Z[1/\Sigma],\Tcal_{\psi})$ by choosing the bottom arrow such that the diagram  
\begin{equation}\nonumber
\xymatrix{\Det^{-1}_{R}\RGamma_{c}(\Z[1/\Sigma],\Tcal)\ar[d]_{-\tenseur_{\psi}\Ocal_{\psi}}\ar[r]^(0.7){\triv_{c}}&R\ar[d]\\
\Det^{-1}_{\Ocal_{\psi}}\RGamma_{c}(\Z[1/\Sigma],\Tcal_{\psi})\ar[r]^(0.7){\triv_{c}}&\Ocal_{\psi}}
\end{equation}
commutes. Moreover, for all $\ell\in\Sigma\backslash\{p\}$ and all pure $\psi$ such that 1 is not an eigenvalue of $\Fr(\ell)$ acting on $M_{\psi}$, the diagram
\begin{equation}\nonumber
\xymatrix{\Lcali_{\ell}(\Fcali)\ar[d]_{-\tenseur_{\psi}\Ocal_{\psi}}\ar[r]^(0.6){\triv_{\Fcali,\ell}}&R\ar[d]\\
\Lcali_{\ell}(\Fcali_{\psi})\ar[r]^(0.6){\triv_{\Fcali_{\psi},\ell}}&\Ocal_{\psi}}
\end{equation}
commutes by theorem \ref{PropLocal}. This data yields a commutative diagram
\begin{equation}\nonumber
\xymatrix{\Lcali(\Fcali)\ar[d]_{-\tenseur_{\psi}\Ocal_{\psi}}\ar[r]^(0.6){\triv_{\Fcali}}&R\ar[d]\\
\Lcali(\Fcali_{\psi})\ar[r]^(0.6){\triv_{\Fcali_{\psi}}}&\Ocal_{\psi}}
\end{equation}
which is an algebraic incarnation of the $p$-adic interpolation of the putative analytic $p$-adic $L$-function of $\Fcali$.
This analysis leads to a refinement of principle \ref{Principle}.
\begin{Principle}\label{PrincipleRef}
The data of the algebraic special values of a $p$-adic family of motives $\Fcali$ unramified outside a finite set $\Sigma$ attached to a $G_{\Q,\Sigma}$-representation $\Tcal$ with coefficients in a domain $R$ is equivalent to the data of the algebraic $p$-adic determinant $\Lcali(\Fcali)$ of $\Fcali$ plus a choice of trivialization of $\Det^{-1}_{R}\RGamma_{c}(\Z[1/\Sigma],\Tcal)$.
\end{Principle}
This principle admits a more precise statement.
\begin{Prop}\label{PropNoObstr}
Let $\Fcali$ be a $p$-adic family of motives unramified outside a finite set $\Sigma$ attached to a $G_{\Q,\Sigma}$-representation $\Tcal$ with coefficients in $R_{\Sigma}$ as in the statement of theorem \ref{PropLocal}. Then the data of an algebraic $p$-adic $L$-function with Euler factors at primes in $\Sigma$ removed and with coefficients in $R_{\Sigma}$ is equivalent to the data of a trivialization of  $\Det^{-1}_{R_{\Sigma}}\RGamma_{c}(\Z[1/\Sigma],\Tcal)$. If $R_{\Sigma}$ is a domain, the data of a trivialization of  $\Det^{-1}_{R_{\Sigma}}\RGamma_{c}(\Z[1/\Sigma],\Tcal)$ is enough to define an algebraic $p$-adic $L$-function without removing Euler factors outside $\{p\}$ provided that $1$ is not eigenvalue of $\Fr(\ell)$ acting on $\Tcal$ for $\ell\in\Sigma\backslash\{p\}$.
\end{Prop}
\begin{proof}
The hypotheses on $R_{\Sigma}$ and $\Tcal$ guarantee that $\Lcali(\Fcali)$ can be defined as an $R_{\Sigma}$-module. The interpolation property with Euler factors removed in general and for the full $p$-adic algebraic determinant is then a consequence of theorem \ref{PropLocal}, proposition \ref{PropSpec} and the discussion above.
\end{proof}

\subsubsection{An example}\label{SubExample}
That the local obstruction of sub-section \ref{SubObstr} is actually realized has been well-known at least since \cite{GreenbergVatsal}.

An explicit example is given by the pair of eigenforms $f_{1}\in S_{2}(\Gamma_{0}(52))$ and $f_{2}\in S_{2}(\Gamma_{0}(364))$ whose attached abelian varieties are respectively the elliptic curves 
\begin{equation}\nonumber
E_{1}:y^{2}=x^{3}+x-10,\ E_{2}:y^{2}=x^{3}-584x+5444.
\end{equation}
The eigenforms $f_{1}$ and $f_{2}$ are congruent modulo 5 so the residual $5$-adic representations attached to $f_{1}$ and $f_{2}$ are isomorphic. The automorphic $\GL_{2}(\Q_{5})$-representations $\pi(f_{1})$ and $\pi(f_{2})$ are both unramified principal series and the Hecke operator at $5$ acts on $f_{1}$ with eigenvalue $2$ and on $f_{2}$ with eigenvalue $-3$ so both are $5$-ordinary.  Let $M(f_{i})=T_{p}(E_{i})(1)$ denote the $5$-adic étale realization of the motive attached to $f_{i}$ and let $\rhobar_{f_{i}}$ be their (isomorphic) residual representations. The representations $\rhobar_{f_{i}}$ are absolutely irreducible. 

The local automorphic $\GL_{2}(\Q_{2})$-representations $\pi(f_{1})_{2}$ and $\pi(f_{2})_{2}$ are both supercuspidal so the local Euler factor of $M(f_{1})$ and $M(f_{2})$ at the prime 2 are both trivial in $\Z_{5}$. The eigenforms $f_{1}$ and $f_{2}$ have the same eigenvalue $-1$ under the action of the Hecke operator at the prime 13 and the local automorphic $\GL_{2}(\Q_{13})$-representations $\pi(f_{1})_{13}$ and $\pi(f_{2})_{13}$ are both unramified Steinberg so both have $1+13^{-1}X$ as local Euler factor. This is a unit in $\Z_{5}$ when evaluated at $X=13^{-s}$ for $s=0$. Finally, the Euler factor $\Eul_{7}(M(f_{1}),X)$ at the prime 7 is equal to $1+(2\cdot 7^{-1})X+(7\cdot 7^{-2})X^{2}$ because $\pi(f_{1})_{7}$ is an unramified principal series with trace of the Hecke operator $T_{7}$ equal to $-2$ and $\Eul_{7}(M(f_{1}),1)$ is zero modulo 5. On the other hand, $\pi(f_{1})_{7}$ is an unramified Steinberg representation with eigenvalue of $T_{7}$ equal to $1$ and so has local Euler factor $\Eul_{7}(M(f_{2}),X)$ equal to $1-7^{-1}X$. This is a unit in $\Z_{5}$ when evaluated at $X=1$. 

It follows from the discussion of sub-section \ref{SubObstr} that the algebraic $L$-value $L_{\{5\}}(M(f_{2}))$ of the motive $M(f_{2})$ has to be a non-unit modulo 5 (as indeed it is). However, the algebraic $L$-value $L_{\{5\}}(M(f_{1}))$ of the motive $M(f_{1})$ at 0 is a unit modulo $5$ and Mazur's question thus admits a negative answer for this pair of congruent eigenforms. 

This result is neither in contradiction with the existence of $p$-adic $L$-functions for families of ordinary (or more generally finite slope) eigenforms nor does it falsify proposition \ref{PropNoObstr}: the fact that $\pi(f_{1})_{7}$ is unramified principal series whereas $\pi(f_{2})_{7}$ is unramified Steinberg implies by the contraposition of theorem \ref{PropLocal} that no $p$-adic family of motives parametrized by a domain can contain both $M(f_{1})$ and $M(f_{2})$. However, $M(f_{1})$ and $M(f_{2})$ belong to a common $p$-adic family of motives parametrized by the reduced nearly-ordinary Hida-Hecke $5$-adic algebra $R^{\ord}_{\Sigma}$ generated by Hecke operators outside $\Sigma=\{2,3,5,7,13\}$ or, more abstractly, by the universal deformation ring $R_{\Sigma}(\rhobar_{f_{i}})$ parametrizing deformations of $\rhobar_{f_{i}}$ unramified outside $\Sigma$ so the possibility of defining an algebraic $p$-adic $L$-function interpolating $L_{\Sigma}(M(f_{1}))$ and $L_{\Sigma}(M(f_{2}))$ with Euler factors at $\{2,3,5,7,13\}$ is equivalent to the data of a trivialization of $\Det^{-1}_{R_{\Sigma}}\RGamma_{c}(\Z[1/\Sigma],\Tcal)$. In that particular case, and in contrast with the general case even of $p$-adic families of automorphic motives, a canonical choice of trivialization is known (and given for instance by Kato's Euler system) so that there exists an algebraic $p$-adic $L$-function interpolating these special values (as well as two algebraic $p$-adic $L$-functions interpolating the special values  $L_{\{5\}}(M(f_{1}))$ and $L_{\{5\}}(M(f_{2}))$ on the largest quotient domains of $R^{\ord}_{\Sigma}$ containing $M(f_{1})$ and $M(f_{2})$ respectively).

\section{The ETNC with coefficients in Hecke rings}\label{SubHecke}
In this sub-section, we fix a reduced complete local noetherian Hecke algebra $R_{\Sigma}$ understood to be the local factor of the $p$-adic Hecke algebra generated by operators outside the finite set of finite primes $\Sigma$ attached to an absolutely irreducible residual $G_{\Q}$-representation. We assume that there exists a $p$-adic family $\Fcali$ parametrized by $R_{\Sigma}$ and attached to a determinant $D$ or, equivalently by assumption \ref{HypIrr}, that there exists a representation 
\begin{equation}\nonumber
\rho:G_{\Q,\Sigma}\fleche\Aut_{R_{\Sigma}}(\Tcal)\simeq\GL_{n}(R_{\Sigma})
\end{equation}
such that $D=\det\circ\rho$. 

Following principle \ref{PrincipleRef} and proposition \ref{PropNoObstr}, we know that the study of algebraic $p$-adic $L$-functions attached to such an $\Fcali$ is equivalent to the study of trivializations of the determinant $\Det^{-1}_{R_{\Sigma}}\RGamma_{c}(\Z[1/\Sigma],\Tcal)$.
\subsection{The problem with the existing conjecture}
In sub-section \ref{SubChange}, a conjecture was outlined for $\Fcali$ assuming it gives rise to a smooth étale sheaf on $\Spec\Z[1/p]$. In this sub-section, we show that this outline, which followed the generalized Iwasawa main conjecture of \cite[Section 3.2]{KatoViaBdR}, does not yield a correct conjecture even for modular forms over $\GL_{2}(\Q)$.

Let $f\in S_{k}(\Gamma_{1}(N))$ be an eigencuspform. Assume for simplicity that $f$ has good ordinary reduction at $p$ and that its residual representation $\rhobar_{f}$ is absolutely irreducible. Let $F_{\pid}$ be a finite extension of $\qp$ containing the eigenvalues of $f$ and let $\Ocal$ be its ring of integers. Denote by $\Lambda_{\Ocal}$ the ring $\Ocal[[\Gamma]]$. We denote by $T(f)$ the $G_{\Q,\Sigma}$-representation with coefficients in $\Ocal$ attached to $f$. Then $T(f)$ is a lattice inside the quotient $\Hun_{\et}(X_{1}(N)\times_{\Q}\Qbar,F_{\pid})(f)$ of $\Hun_{\et}(X_{1}(N)\times_{\Q}\Qbar,F_{\pid})$ on which $T(\ell)\in R_{\Sigma}$ acts as $a_{\ell}(f)$. Let us fix a finite set $\Sigma$ of finite primes containing $p$ and the primes dividing $N$ and let $R_{\Sigma}$ be the local factor attached to $\rhobar_{f}$ of the reduced nearly-ordinary Hida-Hecke algebra generated by Hecke operators outside $\Sigma$ together with the Hecke operate $U_{p}$. We show that the family $\Fcali$ of (motives of) $p$-ordinary eigencuspforms containing $f$ cannot always satisfy the change of coefficients property of sub-section \ref{SubChange}.

Because our intent is to construct a counter-example, we may further assume that each irreducible component of $\Spec R_{\Sigma}$ is isomorphic to $\Spec\zp[[X_{1},X_{2}]]$. This is a mild assumption as the quotient $R_{\Sigma}/\aid$ of $R_{\Sigma}$ by one of its minimal prime, which is a finite and free module over $\zp[[X_{1},X_{2}]]$  in general, is isomorphic to $\zp[[X_{1},X_{2}]]$ provided that for each tame level $M$ such that $\{\ell|M\}\subset\Sigma$ and for each cohomological weight $k$, there is at most one ordinary newform in $S_{k}(\Gamma_{1}(M))$ congruent to $f$ (this is a very common happenstance and is for instance the case for the Hida family attached to the set $\{2,3,5,7,13\}$ containing the forms $f_{1}$ and $f_{2}$ of sub-section \ref{SubExample}). 

\subsubsection{Euler systems and trivializations}
We recall that the motive $M(f)$ of $f$ admits a $p$-adic Euler systems $\{\z(f)_{L}\}_{L}$ indexed by finite abelian extensions of $\Q$ which we view as in \cite[Theorem 12.5]{KatoEuler} as a non-zero morphism
\begin{equation}\nonumber
Z(f):\Hun_{\et}(X_{1}(N)\times_{\Q}\Qbar,\Ocal)^{+}(f)\fleche\Hun_{\et}(\Z[1/p],T(f)\tenseur_{\Ocal}\Lambda_{\Ocal}).
\end{equation}
As we already remarked in \ref{SubCyc}, $\Hun_{\et}(\Z[1/p],T(f)\tenseur_{\Ocal}\Lambda_{\Ocal})$ is under our hypotheses a $\Lambda_{\Ocal}$-module free of rank 1 and $\Hun_{\et}(X_{1}(N)\times_{\Q}\Qbar,\Ocal)^{+}$ is a free $\Ocal$-module of rank 1. Hence, the image of $Z(f)$ in $\Hun_{\et}(\Z[1/p],T(f)\tenseur_{\Ocal}\Lambda_{\Ocal})$ is generated by a single element $\z(f)$ and 
\begin{equation}\nonumber
\Hun_{\et}(\Z[1/p],T(f)\tenseur\Lambda_{\Ocal})/\Lambda_{\Ocal}\cdot\z(f)
\end{equation}
is a torsion $\Lambda_{\Ocal}$-module. In a slight abuse of notation, we keep the symbol $Z(f)$ to denote the morphism
\application{Z(f)}{\Lambda}{\Hun_{\et}(\Z[1/p],T(f)\tenseur_{\Ocal}\Lambda_{\Ocal})}{1}{\ \ \z(f)}
and its lift to a morphism of complexes
\applicationsimple{Z(f)}{\Lambda[-1]}{\RGamma_{\et}(\Z[1/p],T\tenseur_{\Ocal}\Lambda_{\Ocal}).}
In agreement with the weak Leopoldt's conjecture, the cone of the morphism $Z(f)$ is then shown in \cite[Theorem 12.4]{KatoEuler} to be acyclic after tensor product with $\Frac(\Lambda_{\Ocal})$.
\begin{Prop}\label{PropKato}
The determinant $\Det_{\Lambda_{\Ocal}}\Cone Z(f)$ is sent to $\Lambda_{\Ocal}$ through the canonical identification
\begin{equation}\nonumber
\left(\Det_{\Lambda_{\Ocal}}\Cone Z(f)\right)\tenseur_{\Lambda_{\Ocal}}\Frac(\Lambda_{\Ocal})\isocan\Frac(\Lambda_{\Ocal})
\end{equation}
if and only if the ETNC for the motive $M(f)$ is true.
\end{Prop}
\begin{proof}
This is a restatement of theorems 12.4, 12.5 and conjecture 12.10 of \cite{KatoEuler}. 
\end{proof}
Proposition \ref{PropKato} can be informally summed up as claiming that if the ETNC is true, then $\z(f)$ is equal to the zeta element $\z_{\Lambda}(M(f))$. In this opportunity, we remark that not only is the list of motives of rank greater than 1 for which the ETNC is known very thin, in fact even the much weaker statement that, if the ETNC is true, then it is true for a specific known system of motivic elements is known in rank greater than 1 only for Kato's Euler systems. Equivalently, Kato's Euler systems is the only Euler systems for motives of rank greater than 1 for which it is definitely known that it is the motivic Euler system predicted by the ETNC if the ETNC is true.

\subsubsection{Discrepancy of Euler factors}\label{SubDisc}
Let $R(\aid)=R_{\Sigma}/\aid$ be the quotient of $R_{\Sigma}$ by the unique minimal prime $\aid$ contained in the kernel of the specialization of $R_{\Sigma}$ attached to $f$. Let $\Fcali(\aid)$ be the $p$-adic family of motives with coefficients in $R(\aid)$ arising from $\Fcali$ and let $T(\aid)$ be the attached $G_{\Q,\Sigma}$-representation. In view of proposition \ref{PropKato}, to show that \cite[Conjecture 3.2.2]{KatoViaBdR} cannot be correct as stated for motives attached to eigenforms, it is enough to show that it may happen that no basis $\z_{R(\aid)}(\Spec\Z[1/p],\Fcali(\aid))$ of
\begin{equation}\nonumber
\Delta_{R(\aid)}(\Spec\Z[1/p],\Fcali(\aid))=\Det^{-1}_{R(\aid)}\RGamma_{\et}(\Z[1/p],\Fcali(\aid))\tenseur_{R(\aid)}\Det^{-1}_{R(\aid)}\Fcali(\aid)^{+}
\end{equation}
satisfies the following three properties (all predicted by \cite[Conjecture 3.2.2]{KatoViaBdR}).
\begin{enumerate}
\item\label{ItETNC} Let $\psi\in\Hom^{\cl}(R(\aid),\Qbar_{p})$ with values in a discrete valuation ring $S$ be the specialization attached to a classical eigencuspform $f_{\psi}$ whose motives belong to $\Fcali(\aid)$. Let $\Fcali_{\psi}$ be the $p$-adic family of motives parametrized by $\Lambda_{S}=S[[\Gamma]]$ equal to the cyclotomic deformation of the $G_{\Q,\Sigma}$-representation of $T(f_{\psi})$. Then $\z_{R(\aid)}(\Spec\Z[1/p],\Fcali(\aid))$ is sent to the basis $\z_{\Lambda_{S}}(\Spec\Z[1/p],\Fcali_{\psi})$ through the canonical isomorphism
\begin{equation}\nonumber
\Delta_{R(\aid)}(\Spec\Z[1/p],\Fcali(\aid))\tenseur_{R(\aid),\psi}\Lambda_{S}\isocan\Delta_{\Lambda_{S}}(\Spec\Z[1/p],\Fcali_{\psi}).
\end{equation}
\item\label{ItInter} Let $\psi\in\Hom^{\cl}(R(\aid),\Qbar_{p})$ with values in a discrete valuation ring $S$ be the specialization attached to a classical eigencuspform $f_{\psi}$ whose motives belong to $\Fcali(\aid)$. Let $\Fcali_{\psi}$ be the $p$-adic family of motives parametrized by $\Lambda_{S}=S[[\Gamma]]$ equal to the cyclotomic deformation of the $G_{\Q,\Sigma}$-representation of $T(f_{\psi})$. Then $\z_{R(\aid)}(\Spec\Z[1/p],\Fcali(\aid))$ is sent to the basis 
\begin{equation}\nonumber
\Det_{\Lambda_{S}}\Cone Z(f_{\psi})\subset\Delta_{\Lambda_{S}[1/p]}(\Spec\Z[1/p],\Fcali_{\psi})
\end{equation}
through the canonical isomorphism
\begin{equation}\nonumber
\Delta_{R(\aid)}(\Spec\Z[1/p],\Fcali(\aid))\tenseur_{R(\aid),\psi}\Lambda_{S}[1/p]\isocan\Delta_{\Lambda_{S}[1/p]}(\Spec\Z[1/p],\Fcali_{\psi}).
\end{equation}
\item\label{ItProb} Let $\psi\in\Hom(R(\aid),\Qbar_{p})$ with values in a discrete valuation ring $S$ be a specialization of $R(\aid)$. Let $\Fcali_{\psi}$ be the étale sheaf parametrized by $\Lambda_{S}=S[[\Gamma]]$ equal to the cyclotomic deformation of the $G_{\Q,\Sigma}$-representation of $\Tcal\tenseur_{R(\aid),\psi}S$. Then $\z_{R(\aid)}(\Spec\Z[1/p],\Fcali(\aid))$ is sent to a basis of
\begin{equation}\nonumber
\Delta_{\Lambda_{S}[1/p]}(\Spec\Z[1/p],\Fcali_{\psi})
\end{equation}
through the canonical isomorphism
\begin{equation}\nonumber
\Delta_{R(\aid)}(\Spec\Z[1/p],\Fcali(\aid))\tenseur_{R(\aid),\psi}\Lambda_{S}[1/p]\isocan\Delta_{\Lambda_{S}[1/p]}(\Spec\Z[1/p],\Fcali_{\psi}).
\end{equation}
\end{enumerate}
In other words, we claim that the $p$-adic interpolation property at motivic points together with the truth of the ETNC at these points (properties \ref{ItETNC} and \ref{ItInter}) may in general contradict the $p$-adic interpolation property at all points (property \ref{ItProb}). As we of course believe in the truth of the ETNC for modular forms, we expect that it is property \ref{ItProb} which may fail.

Property \ref{ItInter} above is a statement of $p$-adic interpolation of the ETNC at motivic points. As the ETNC with coefficients in the cyclotomic Iwasawa algebra is known to be equivalent to conjectures in Iwasawa theory of \cite{GreenbergIwasawaRepresentation,GreenbergIwasawaMotives} for good ordinary modular forms, constructing a basis of $\Delta_{R(\aid)}(\Spec\Z[1/p],\Fcali(\aid))$ satisfying property \ref{ItInter} amounts to constructing such a basis compatible with these conjectures of Greenberg. The following proposition consequently shows that the existence of a $p$-adic $L$-function defined on $\Spec R(\aid)$ in the sense of \cite{EmertonPollackWeston} is enough to construct a basis of $\Delta_{R(\aid)}(\Spec\Z[1/p],\Fcali(\aid))$ satisfying property \ref{ItInter}.
\begin{Prop}\label{PropL}
Under our hypotheses on $\Fcali$, there exists a basis of $\Delta_{R(\aid)}(\Spec\Z[1/p],\Fcali(\aid))$ satisfying property \ref{ItInter}.
\end{Prop} 
\begin{proof}
As we have assumed that $R(\aid)$ is a Gorenstein ring (even a regular one), the $p$-adic $L$-function $L(\aid)\in R(\aid)$ of \cite[Section 3.4]{EmertonPollackWeston}, which is defined in \textit{loc. cit.} under slightly different hypotheses on the residual representation, exists thanks to \cite{KitagawaLpadic}. 

We first consider the \Nekovar-Selmer complex $\RGamma_{f}(G_{\Q,\Sigma},T(\aid))$ which is well-defined as $T(\aid)$ is nearly-ordinary. By \cite{OchiaiMainConjecture,FouquetOchiai}, there is a canonical isomorphism 
\begin{equation}\nonumber
\Det^{-1}_{R(\aid)}H^{2}_{f}(G_{\Q,\Sigma},T(\aid))\tenseur_{R(\aid),\psi}\Lambda_{\psi}\isocan\Det^{-1}_{\Lambda_{\psi}}H^{2}_{f}(G_{\Q,\Sigma},T_{\psi}\tenseur_{\Ocal_{\psi}}\Lambda_{\psi})
\end{equation}
for all specialization $\psi$ attached to a modular form $f_{\psi}$ with eigenvalues in $\Ocal_{\psi}$ (here $\Lambda_{\psi}$ denotes the ring $\Ocal_{\psi}[[\Gamma]]$) and the modules $H^{2}_{f}(G_{\Q,\Sigma},T(\aid))$ and $H^{2}_{f}(G_{\Q,\Sigma},T_{\psi}\tenseur_{\Ocal_{\psi}}\Lambda_{\psi})$ are torsion as $R(\aid)$-module and $\Lambda_{\psi}$-module respectively. Hence the trivialization
\begin{equation}\nonumber
\triv_{f}:\Det^{-1}_{R(\aid)}H^{2}_{f}(G_{\Q,\Sigma},T(\aid))\subset\Det^{-1}_{R(\aid)}H^{2}_{f}(G_{\Q,\Sigma},T(\aid))\tenseur_{R(\aid)}\Frac(R(\aid))\isocan\Frac(R(\aid))
\end{equation}
of $\Det^{-1}_{R(\aid)}H^{2}_{f}(G_{\Q,\Sigma},T(\aid))$ is compatible with classical specializations in the sense that the diagram 
\begin{equation}\nonumber
\xymatrix{
\Det^{-1}_{R(\aid)}H^{2}_{f}(G_{\Q,\Sigma},T(\aid))\ar[r]\ar[d]&R(\aid)\ar[d]\\
\Det^{-1}_{\Lambda_{\psi}}H^{2}_{f}(G_{\Q,\Sigma},T_{\psi}\tenseur\Lambda_{\psi})\ar[r]&\Lambda_{\psi}
}
\end{equation}
commutes.

We next construct a trivialization of the determinant
\begin{equation}\nonumber
\Det^{-1}_{R(\aid)}\RGamma_{f}(G_{\Q,\Sigma},T(\aid))\tenseur_{R(\aid)}\Det^{-1}_{R(\aid)}\Fcali^{+}.
\end{equation}
To that effect, we choose a basis $b_{1}$ of the graded invertible module $\Det_{R(\aid)}\Hun_{f}(G_{\Q,\Sigma},T(\aid))$ and a basis  $b_{2}$ of the graded invertible module $\Det^{-1}_{R(\aid)}\Fcali(\aid)^{+}$. The map
\application{\pi}{\Det^{-1}_{R(\aid)}\Fcali(\aid)^{+}}{\Det_{R(\aid)}\Hun_{f}(G_{\Q,\Sigma},T(\aid))}{b_{2}}{L(\aid)\cdot b_{1}}
induces a trivialization of
\begin{equation}\nonumber
X=\Det_{R(\aid)}\Hun_{f}(G_{\Q,\Sigma},T(\aid))\tenseur\Det^{-1}_{R(\aid)}\Fcali(\aid)^{+}\plonge\Frac(R(\aid))
\end{equation}
sending a basis of $X$ to $L(\aid)$. Combined with the trivialization $\triv_{f}$ of $\Det^{-1}_{R(\aid)}H^{2}_{f}(G_{\Q,\Sigma},T(\aid))$ as above yields 
\begin{equation}\nonumber
\triv:\Det^{-1}_{R(\aid)}\RGamma_{f}(G_{\Q,\Sigma},T(\aid))\tenseur_{R(\aid)}\Det^{-1}_{R(\aid)}\Fcali^{+}\fleche\Frac(R(\aid)).
\end{equation}
In fact, under our ongoing hypothesis that $R(\aid)$ is a regular local ring, we can compute that the image of a basis of the right-hand side is sent to a basis of the invertible ideal
\begin{equation}\nonumber
\frac{\carac_{R(\aid)}H^{2}_{f}(G_{\Q,\Sigma},T(\aid))}{L(\aid)}\subset\Frac(R(\aid)).
\end{equation}
The property of $H^{2}_{f}(G_{\Q,\Sigma},T(\aid))$ recalled above and the interpolation property of $L(\aid)$ given in \cite[Proposition 3.4.3]{EmertonPollackWeston} imply that $\triv$ is compatible with classical specialization in the sense that the diagram
\begin{equation}\nonumber
\xymatrix{
\Det^{-1}_{R(\aid)}\RGamma_{f}(G_{\Q,\Sigma},T(\aid))\tenseur_{R(\aid)}\Det^{-1}_{R(\aid)}\Fcali^{+}\ar[r]\ar[d]&R(\aid)\ar[d]\\
\Det^{-1}_{\Lambda_{\psi}}\RGamma_{f}(G_{\Q,\Sigma},T_{\psi}\tenseur\Lambda_{\psi})\tenseur_{\Lambda_{\psi}}\Det^{-1}_{\Lambda_{\psi}}\Fcali_{\psi}^{+}\ar[r]&\Lambda_{\psi}
}
\end{equation}
commutes.

By \cite[Section 17.3]{KatoEuler}, the image of the trivialization 
\begin{equation}\nonumber
\triv:\Det^{-1}_{\Lambda_{\psi}}\RGamma_{f}(G_{\Q,\Sigma},T_{\psi}\tenseur\Lambda_{\psi})\tenseur_{\Lambda_{\psi}}\Det^{-1}_{\Lambda_{\psi}}\Fcali_{\psi}^{+}\fleche\Lambda_{\psi}
\end{equation}
for a specialization $\psi$ attached to an eigencuspform with good ordinary specialization is equal to the image of $\Det_{\Lambda_{\psi}}Z(f_{\psi})$ through the canonical trivialization of proposition \ref{PropKato}. Hence the diagram 
\begin{equation}\nonumber
\xymatrix{
\Det^{-1}_{R(\aid)}\RGamma_{f}(G_{\Q,\Sigma},T(\aid))\tenseur_{R(\aid)}\Det^{-1}_{R(\aid)}\Fcali^{+}\ar[r]\ar[d]&R(\aid)\ar[d]\\
\Det^{-1}_{\Lambda_{\psi}}\RGamma_{f}(G_{\Q,\Sigma},T_{\psi}\tenseur\Lambda_{\psi})\tenseur_{\Lambda_{\psi}}\Det^{-1}_{\Lambda_{\psi}}\Fcali_{\psi}^{+}\ar[r]&\Lambda_{\psi}
}
\end{equation}
in which the bottom horizontal map is the trivialization of proposition \ref{PropKato} and the top horizontal map sends a chosen basis of $\Delta_{R(\aid)}(\Spec\Z[1/p],\Fcali)$ to a basis of the image of $\triv$ is commutative and there thus exists a basis of $\Delta_{R(\aid)}(\Spec\Z[1/p],\Fcali)$ satisfying property \ref{ItInter} as claimed.
\end{proof}
A more explicit and natural (but also much more intricate and subtle) proof of the existence of a basis of $\Delta_{R(\aid)}(\Spec\Z[1/p],\Fcali)$ satisfying property \ref{ItInter} can be given by $p$-adically interpolating Kato's Euler systems in Hida families. See \cite[Theorem 3.2.3]{FukayaKatoSharifi} for a proof along these lines.
\begin{TheoEnglish}\label{TheoFailure}
There exist $p$-adic families for which the three requirements \ref{ItETNC}, \ref{ItInter} and \ref{ItProb} cannot be simultaneously true.
\end{TheoEnglish}
\begin{proof}
As at most one basis can satisfy requirement \ref{ItInter}, proposition \ref{PropL} and its proof imply that there exist families for which the requirements \ref{ItETNC} and \ref{ItInter}, if they are satisfied, have to be satisfied by the basis arising from the trivialization constructed from the $p$-adic $L$-function $L(\aid)$. Consider however a prime $\ell\in\Sigma$ at which $\Fcali$ is ramified and a (necessarily non-classical) point $\Fcali_{\psi}$ such that $T_{\psi}$ is unramified at $\ell$. The element $L(\aid)$ of \cite{EmertonPollackWeston} is the image through the isomorphism of \cite[Theorem 3.6.2]{EmertonPollackWeston} of a $p$-adic $L$-function interpolating the special values at motivic points on the whole of $\Spec R_{\Sigma}$ but with the generic Euler factors with coefficients in $R(\aid)$ at $\Sigma$ removed. Hence $\psi(L(\aid))$ is the product of a $p$-adic $L$-function times $\psi(\Eul_{\ell}(\Fcali,1))$. However, the trivialization of $\Det^{-1}_{R(\aid)}\RGamma_{f}(G_{\Q,\Sigma},T(\aid))$ composed with $\psi$ yields by construction the Euler factor $\Eul_{\ell}(\Fcali_{\psi},1)$. As $\psi$ is not classical, there is no reason for these two Euler factors to coincide and one could very well be a unit without the other being one. If this happens, a basis of $\Delta(\Spec\Z[1/p],T(\aid))$ satisfying the requirements \ref{ItETNC} and \ref{ItInter} cannot satisfy the requirement \ref{ItProb} at $\psi$.
\end{proof}
No explicit example of Hida families of nearly-ordinary eigencuspforms is known to exhibit the phenomenon described at the end of the proof of theorem \ref{TheoFailure}. However, this is mainly due to the extreme computational difficulty of computing non-classical points in Hida families. The consensus among the experts is that such families exist (and in fact that they are the generic case). 
\subsection{A revised conjecture}
The detailed example of sub-section \ref{SubDisc} somewhat obscures the general principle underlying all such counterexamples: in agreement with principle \ref{PrincipleRef}, we should expect $p$-adic $L$-function to come from trivialization of the compactly supporte étale cohomology over the whole reduced space together with specialization maps to irreducible components recovering the omitted local Euler factors. In our actual counterexample, this was indicated by the crucial step in the proof of theorem \ref{TheoFailure} where it was noted that the $p$-adic $L$-function of \cite{EmertonPollackWeston} is the image of a more general $p$-adic $L$-function with Euelr factors removed with coefficients in $R_{\Sigma}$. By construction, the compactly supported étale cohomology complex will involve \emph{Galois} Euler factor in these specialization maps (namely the algebraic determinants $\Lcali_{\ell}(\Fcali)$) whereas the trivialization, because it involves Betti cohomology, will involve \emph{automorphic} Euler factors. At classical points, these Euler factors coincide according to the Local Langlands Correspondence and local-global compatibility but they have no reason to coincide in general. The following conjectures are designed to remedy this problem.
\subsubsection{Statements of the conjectures}
Let $\Sigma$ be a finite set of finite primes containing $\{p\}$. Let $R_{\Sigma}$ be a reduced $p$-torsion free complete local noetherian ring with finite residue field of characteristic $p$. Let $\Fcali$ be the $p$-adic family of automorphic motives attached to an $n$-dimensional determinant $D$ with values in $R_{\Sigma}$. Assume that $\Fcali$ is unramified outside $\Sigma$, that its residual representation $\rhobar$ is absolutely irreducible and denote by $(T_{\Sigma},\rho_{\Sigma},R_{\Sigma})$ the $G_{\Q,\Sigma}$-representation such that $D=\det\circ\rho$. For $\aid$ a minimal prime ideal in $\Spec R_{\Sigma}$, denote by $\Fcali(\aid)$ the specialized $p$-adic family and by $(T(\aid),\rho(\aid),R(\aid))$ the corresponding $G_{\Q,\Sigma}$-representation. If $\psi\in\Hom(R_{\Sigma},\Qbar_{p})$ has values in a discrete valuation ring $\Ocal_{\psi}$, denote likewise by $\Fcali_{\psi}$ the specialized family.

\begin{Conj}\label{ConjMain}
There exist a perfect complex of $R_{\Sigma}$-module $M_{\Sigma}$ and a trivialization 
\begin{equation}\nonumber
\triv:\Det^{-1}_{R_{\Sigma}}\RGamma_{c}(\Z[1/\Sigma],T_{\Sigma})\tenseur_{R_{\Sigma}}\Det^{-1}_{R_{\Sigma}}M_{\Sigma}\fleche R_{\Sigma}
\end{equation}verifying the following properties.
\begin{enumerate}
\item For all minimal prime ideal $\aid\in\Spec R_{\Sigma}$ and all $\psi\in\Hom(R_{\Sigma},\Qbar_{p})$ denote by $M_{\Sigma}\tenseur_{R_{\Sigma}}R(\aid)$ by $M(\aid)$ and $M_{\Sigma}\tenseur_{R_{\Sigma},\psi}\Ocal_{\psi}$ by $M_{\psi}$. Then there are specified isomorphisms 
\begin{equation}\nonumber
\left(\Det^{-1}_{R_{\Sigma}}\RGamma_{c}(\Z[1/\Sigma],T_{\Sigma})\tenseur_{R_{\Sigma}}\Det^{-1}_{R_{\Sigma}}M_{\Sigma}\right)\tenseur_{R_{\Sigma}}R(\aid)\isocan\Lcali_{\{p\}}(\Fcali(\aid))\tenseur_{R(\aid)}\Det^{-1}_{R(\aid)}M(\aid)
\end{equation}
and 
\begin{equation}\nonumber
\left(\Det^{-1}_{R_{\Sigma}}\RGamma_{c}(\Z[1/\Sigma],T_{\Sigma})\tenseur_{R_{\Sigma}}\Det^{-1}_{R_{\Sigma}}M_{\Sigma}\right)\tenseur_{R_{\Sigma},\psi}\Ocal_{\psi}\isocan\Lcali_{\{p\}}(\Fcali_{\psi})\tenseur_{\Ocal_{\psi}}\Det^{-1}_{\Ocal_{\psi}}M_{\psi}
\end{equation}
inducing diagrams
\begin{equation}\label{DiagTriv}
\xymatrix{
\Det^{-1}_{R_{\Sigma}}\RGamma_{c}(\Z[1/\Sigma],T_{\Sigma})\tenseur_{R_{\Sigma}}\Det^{-1}_{R_{\Sigma}}M_{\Sigma}\ar[r]^(0.8){\triv}\ar[d]&R_{\Sigma}\ar[d]\\
\Lcali_{\{p\}}(\Fcali(\aid))\tenseur_{R(\aid)}\Det^{-1}_{R(\aid)}M(\aid)\ar[r]&R(\aid)}
\end{equation}
and
\begin{equation}\label{DiagTrivPsi}
\xymatrix{
\Det^{-1}_{R_{\Sigma}}\RGamma_{c}(\Z[1/\Sigma],T_{\Sigma})\tenseur_{R_{\Sigma}}\Det^{-1}_{R_{\Sigma}}M_{\Sigma}\ar[r]^(0.8){\triv}\ar[d]&R_{\Sigma}\ar[d]\\
\Lcali_{\{p\}}(\Fcali_{\psi})\tenseur_{\Ocal_{\psi}}\Det^{-1}_{\Ocal_{\psi}}M_{\psi}\ar[r]&\Ocal_{\psi}}
\end{equation}
where the horizontal maps of diagram \ref{DiagTriv} are defined by commutativity.
\item There exists a basis $\z_{\Sigma}$ of $\Det^{-1}_{R_{\Sigma}}\RGamma_{c}(\Z[1/\Sigma],T_{\Sigma})\tenseur_{R_{\Sigma}}\Det^{-1}_{R_{\Sigma}}M_{\Sigma}$ such that the image of $\z_{\Sigma}$ through a motivic specialization $\psi$ with values in $\Ocal_{\psi}$ is the $p$-adic étale realization $\z_{\{p\}}(M_{\psi})\in\Delta_{\{p\}}(M_{\psi})$ of the zeta element of $M_{\psi}$.
\end{enumerate}
\end{Conj}
In the context of the next proposition, we consider that a basis $\z_{\{p\}}(M)$ of the fundamental line $\Delta_{\{p\}}(M)$ is a motivic element if it satisfies the first property of conjecture \ref{ConjTNC} and that the conjecture is then that it satisfies the second. 
\begin{Prop}\label{Prop}
If conjecture \ref{ConjMain} is true, then question \ref{Q3} admits a positive answer.
\end{Prop}
\begin{proof}
Assume conjecture \ref{ConjMain} to be true for the family $\Fcali$. Let $M_{\psi}$ and $M_{\phi}$ be two motives with coefficients in $\Ocal_{\psi}$ and $\Ocal_{\phi}$ respectively and both belonging to $\Fcali$. Assume that the TNC is true for $M_{\phi}$. Then the image of $\triv$ is an element of $R_{\Sigma}$ which is sent to $\Ocal$ by the local morphisms of the arrows of diagram \eqref{DiagTrivPsi} for $M_{\phi}$. Thus, $\triv$ has to be an isomorphism. The image of $\Lcali_{\{p\}}(\Fcali_{\psi})\tenseur_{\Ocal_{\psi}}\Det^{-1}_{\Ocal_{\psi}}M_{\psi}$ through the horizontal map of diagram \eqref{DiagTrivPsi} for $\psi$ then has to be $\Ocal_{\psi}$ and so the TNC is true for $M_{\psi}$.
\end{proof}
The right vertical arrows of diagram \eqref{DiagTriv} and \eqref{DiagTrivPsi} involve a comparison of
\begin{equation}\nonumber
\Det^{-1}_{R_{\Sigma}}\RGamma_{c}(\Z[1/\Sigma],T_{\Sigma})\tenseur_{R_{\Sigma}}R(\aid)\textrm{ and }\Det^{-1}\RGamma_{c}(\Z[1/\Sigma],T_{\Sigma})\tenseur_{R_{\Sigma}}\Ocal_{\psi}
\end{equation}
with $\Lcali(\Fcali(\aid))$ and $\Lcali(\Fcali_{\psi})$ respectively so involve the tensor products 
\begin{equation}\nonumber
\produittenseur{\ell\in\Sigma}{}\Lcali_{\ell}(\Fcali(\aid))\textrm{ and }\produittenseur{\ell\in\Sigma}{}\Lcali_{\ell}(\Fcali_{\psi}).
\end{equation}
Conjecture \ref{ConjMain} can thus be true only if the same tensor products appear with an opposite sign in the maps
\begin{equation}\nonumber
\Det^{-1}_{R_{\Sigma}}\fleche\Det^{-1}_{R(\aid)}M(\aid),\ \Det^{-1}_{R_{\Sigma}}\fleche\Det^{-1}_{\Ocal_{\psi}}M_{\psi}
\end{equation}
so that $M_{\Sigma}$ somehow has to know about the Euler factors of $\Fcali$ and $\Fcali_{\psi}$. Because $\Lcali_{\ell}(\Fcali)\tenseur_{R(\aid)}\Ocal_{\psi}$ need not be canonically isomorphic to $\Lcali_{\ell}(\Fcali_{\psi})$ when $\psi$ is not classical, we further remark that if conjecture \ref{ConjMain} is true, then the map 
\begin{equation}\nonumber
\Det^{-1}_{R_{\Sigma}}\fleche\Det^{-1}_{\Ocal_{\psi}}M_{\psi}
\end{equation}
does not in general factor through 
\begin{equation}\nonumber
\Det^{-1}_{R_{\Sigma}}\fleche\Det^{-1}_{R(\aid)}M(\aid)
\end{equation}
and thus has to be sensitive not only to the irreducible component of $\Spec R_{\Sigma}$ through which $\psi$ factors but also to the respective role of the other components. On the other hand, local Euler factors do commute with specialization when $\psi$ is classical so the two diagrams \eqref{DiagTriv} and \eqref{DiagTrivPsi} should combine in that case in a commutative diagram as follows.
\begin{equation}\nonumber
\xymatrix{
\Det^{-1}_{R_{\Sigma}}\RGamma_{c}(\Z[1/\Sigma],T_{\Sigma})\tenseur_{R_{\Sigma}}\Det^{-1}_{R_{\Sigma}}M_{\Sigma}\ar[r]^(0.8){\triv}\ar[d]&R_{\Sigma}\ar[d]\\
\Lcali_{\{p\}}(\Fcali(\aid))\tenseur_{R(\aid)}\Det^{-1}_{R(\aid)}M(\aid)\ar[d]\ar[r]&R(\aid)\ar[d]\\
\Lcali_{\{p\}}(\Fcali_{\psi})\tenseur_{\Ocal_{\psi}}\Det^{-1}_{\Ocal_{\psi}}M_{\psi}\ar[r]&\Ocal_{\psi}
}
\end{equation}
Not so many perfect complexes $M_{\Sigma}$ can satisfy such a list of requirements so we are led to the following conjecture.

A tower of Shimura variety $(X(U_{i}))_{i}$ is a projective system $(X(U_{i}))_{i}$ of Shimura varieties indexed by the level such that $U_{i}$ is maximal hyperspecial at primes outside $\Sigma$, constant outside $p$ and shrinks to identity at $p$.
\begin{Conj}\label{ConjCompleted}
Assume there exists a tower of Shimura varieties such that all motivic points of $\Fcali$ arise as the localization at a suitable maximal ideal of the Hecke algebra $R_{\Sigma}$ of the cohomology $H^{d}_{\et}(X(U_{i})\times_{\Q}\Qbar,\Fcal)$ in degree $d$ of the Shimura variety $X(U_{i})$ with coefficients in some $p$-adic local system $\Fcal$. Let $n^{+}$ be the rank of $T_{\Sigma}^{+}$ and let $n^{-}$ be $n-n^{+}$. Then there exist special elements $(\gamma^{+}_{s})_{1\leq s\leq n^{+}}$ and $(\gamma^{-}_{s})_{1\leq s\leq n^{-}}$ in the completed cohomology
\begin{equation}\nonumber
\Htilde^{d}_{\et}(X_{\Sigma},\qp)_{\mgot_{\rhobar}}=\left(\limproj{r}\ \liminj{i}\ H^{d}_{\et}(X(U_{i})\times_{\Q}\Qbar,\Z/p^{r}\Z)_{\mgot_{\rhobar}}\right)\tenseur_{\zp}\qp
\end{equation}
localized at the maximal ideal of the Hecke algebra corresponding to the residual representation $\rhobar$ such that $M_{\Sigma}$ is a projective resolution of the $R_{\Sigma}$-span of the $(\gamma^{\pm}_{s})$ in $\Htilde^{d}_{\et}(X_{\Sigma},\zp)_{\mgot_{\rhobar}}$.
\end{Conj}
\subsubsection{Remarks and discussion}
We explain the motivations underlying conjectures \ref{ConjMain} and \ref{ConjCompleted}.

First of all, they express together the common insight that while the trivialization of the algebraic determinant attached to a single motive should come from its Betti cohomology, the trivialization of the algebraic determinant of a $p$-adic family of motives should come from the cohomology of the relevant object at all possible levels; that is to say completed cohomology. In that sense, they are natural generalizations of the statement that zeta elements should form Euler systems: just like the compatibility of the TNC with respect to proper base-change implies that zeta elements for individual motives should arise as specializations of zeta elements with coefficients in Iwasawa-theoretic families encapsulating all abelian fields of definition (as recalled in sub-section \ref{SubCyc}), its compatibility with change of ring of coefficients suggests that zeta elements should come from elements in completed cohomology.

Furthermore, the hope is that a module $M_{\Sigma}$ as in conjecture \ref{ConjCompleted} inherits from $\Htilde^{d}_{\et}(X_{\Sigma},\qp)_{\mgot_{\rhobar}}$ the conjectural property that this object has of interpolating the Breuil-Schneider Local Langlands Correspondence of \cite{BreuilSchneider} at primes outside $p$ and thus that $M_{\Sigma}$ has the property described above with respect to specialization maps (to wit, that the correct Euler factors appear in specializations, including when they are not classical).

It would of course be very desirable to give even a conjectural description of the elements $\gamma^{\pm}$ of conjecture \ref{ConjCompleted}. However, as the images of these elements in the étale cohomology complex should coincide with the bases of zeta elements at classical points, these images form an Euler system and the problem of constructing them even conjecturally amounts to the problem of giving a conjectural construction of Euler systems in completed cohomology of Shimura varieties: at the very least an arduous one. A first step towards characterizing them more precisely would be to extend the generalized reciprocity law of \cite{KatoReciprocity} to eigenvarieties. In fact, the proof of this explicit reciprocity law for modular forms in \cite{KatoEuler} requires crucially the study of the Shimura curve with infinite level $X_{1}(Np^{\infty})$, what is nowadays called a perfectoid Shimura variety. The first named author confesses that the almost-isomorphism between completed cohomology and the cohomology of perfectoid Shimura varieties of \cite[Theorem IV.2.1]{ScholzeTorsion} evoked in him the dream that the elements $\gamma_{s}^{\pm}$ and the module $M_{\Sigma}$ could be constructed by un-tilting a comparable module over a perfectoid Shimura variety of characteristic $p$ (such a module is easy to construct because conjectures on special values of $L$-functions are known in great generality even when the residual characteristic of the scheme and the residual characteristic of the étale sheaf are the same; see for instance \cite{EmertonKisin}). However, such a dream remains at present largely just that-a dream-as the almost-isomorphism and the tilting process in question are highly transcendental so that tracking algebraic special elements through such constructions seems hard.

\subsubsection{Supporting evidence}
In this sub-section, we review the main results of \cite{HeckeETNC} and explain why they are in agreement with conjectures \ref{ConjMain} and \ref{ConjCompleted} for $\G=\GL_{2}$.

Let $f\in S_{k}(\Gamma_{1}(N))$ be an eigencuspform of weight $k\geq 2$ and let $p\nmid N$ be an odd prime. Denote by
\begin{equation}\nonumber
\rhobar_{f}:G_{\Q}\fleche\GL_{2}(\bar{\Fp})
\end{equation}
the residual $G_{\Q}$-representation attached to $f$ and by $N(\rhobar_{f})$ its Artin conductor outside $p$. Let $\Sigma$ be a finite set of finite places containing $\{\ell|N(\rhobar_{f})p\}$.
\begin{TheoEnglish}\label{TheoHecke}
 Assume that $\rhobar_{f}$ satisfies the following hypotheses.
\begin{enumerate}
\item Let $p^{*}$ be $(-1)^{(p-1)/2}p$. The representation $\rhobar_{{f}}|_{G_{\Q(\sqrt{p^{*}})}}$ is absolutely irreducible.
\item Either the representation $\rhobar_{{f}}|_{G_{\qp}}$ is reducible but not scalar (in which case we say that $\rhobar_{f}$ is nearly-ordinary) or there exists a commutative finite flat $p$-torsion group scheme $G$ over $\zp$ and a character $\bar{\mu}$ such that $\rhobar_{{f}}\tenseur\bar{\mu}^{-1}$ is isomorphic as $\bar{\Fp}[G_{\qp}]$-module to $(G\times_{\zp}\Qbar_{p})[p]$ (in which case we say that $\rhobar_{f}$ is flat).
\item\label{HypMonodromy} There exists $\ell\in\Sigma$ such that $\ell||N(\rhobar_{f})$ and such that the image of $\rhobar_{f}|_{G_{\Q_{\ell}}}$ contains a non-identity unipotent element.
\end{enumerate}
Let $R_{\Sigma}$ be the universal deformation ring parametrizing deformations of $\rhobar_{f}$ which are nearly-ordinary or flat (according to whether $f$ itself is nearly-ordinary or flat) and which are unramified outside $\Sigma$ (under our hypotheses, this is a complete intersection ring coinciding with an appropriate Hecke algebra). Let $T_{\Sigma}$ be the $G_{\Q,\Sigma}$-representation with coefficients in $R_{\Sigma,\Iw}$ deforming $\rhobar_{f}$. Then there exists a module $M_{\Sigma}$ in the completed cohomology of the tower of modular curves as in conjecture \ref{ConjCompleted} such that the fundamental line 
\begin{equation}\nonumber
\Delta_{\Sigma}(T_{\Sigma})=\Det^{-1}_{R_{\Sigma}}\RGamma_{c}(G_{\Q,\Sigma},T_{\Sigma})\tenseur_{R_{\Sigma}}\Det_{R_{\Sigma}}^{-1}M_{\Sigma}
\end{equation}
satisfies the following properties.
\begin{enumerate}
\item There exists a trivialization
\begin{equation}\nonumber
\triv:\Delta_{\Sigma}(T_{\Sigma})\fleche R_{\Sigma}
\end{equation} satisfying the first property of conjecture \ref{ConjMain} and such that the second property is satisfied at a classical point $\psi$ if and only if the ETNC is true for $f_{\psi}$.
\item The trivialization $\triv$ is an isomorphism (equivalently the ETNC with coefficients in $R_{\Sigma}$ is true) if and only if the ETNC for $f_{\psi}$ is true for all classical specializations of $R_{\Sigma}$ if and only if there exists one classical specialization $\psi$ of $R_{\Sigma}$ such that the ETNC is true for $f_{\psi}$.
\item If moreover the semi-simplification of $\rhobar_{{f}}|_{G_{\qp}}$ is reducible but not scalar and if $f$ belongs to $S_{k}(\Gamma_{1}(p^{r}))\cap S_{k}(\Gamma_{0}(N))$, then $\triv$ is an isomorphism; so the ETNC with coefficients in $R_{\Sigma}$ is true and the ETNC is true for all classical specializations of $R_{\Sigma}$.
\end{enumerate}
\end{TheoEnglish}
We refer to \cite{HeckeETNC} for the proof and content ourselves with noting that it proceeds in the following way: first, the compatibilities of conjecture \ref{ConjMain} are proved; then, they are used conjointly with the method of Taylor-Wiles system to reduce to the case of a regular local ring of coefficients and then finally the method of Euler systems is used to establish the result.
\subsection*{Appendix}
\subsubsection*{The determinant functor}
Let $A$ be a ring. A graded invertible $A$-module $(P,r)$ is a pair formed with a projective $A$-module $P$ of rank one and a locally constant map $r$ from $\Spec A$ to $\Z$. A perfect complex of $A$-modules is a complex of $A$-modules quasi-isomorphic to a bounded complex of projective $A$-modules of finite types. An $A$-module is perfect if it admits a resolution by a perfect complex, hence if it has finite projective dimension over $A$. If $A$ is a local noetherian ring, then all bounded complexes of $A$-modules are perfect if (and only if) $A$ is regular by the theorem of Auslander-Buchsbaum and Serre. 

The determinant functor $\Det_{\Lambda}(-)$ of \cite{MumfordKnudsen} (see also \cite{DeligneDeterminant}) is a functor
\begin{equation}\nonumber
\Det_{\Lambda}P=\left(\underset{\Lambda}{\overset{{\rank_{\Lambda}P}}{\bigwedge}}P,\rank_{\Lambda}P\right)
\end{equation} from the category of projective $\Lambda$-modules (with morphisms restricted to isomorphisms) to the symmetric monoidal category of graded invertible $\Lambda$-modules (with morphisms restricted to isomorphisms). The determinant functor admits an extension (which is unique up to canonical isomorphism provided it satisfies the requirements of \cite[Definition I]{MumfordKnudsen}) to a functor from the category of perfect complexes of $\Lambda$-modules with morphisms restricted to quasi-isomorphisms to the category of graded invertible $\Lambda$-modules by setting
\begin{equation}
\Det_{\Lambda}C^{\bullet}=\underset{i\in\Z}{\bigotimes}{}\Det_{\Lambda}^{(-1)^{i}}C^{i}.
\end{equation}
In particular, $\Det_{\Lambda}(-)$ satisfies the following properties: it commutes with derived tensor product, there is a canonical isomorphism between $\Det_{\Lambda}(0)$ and $(\Lambda,0)$ and there exists a canonical isomorphism
\begin{equation}\label{EqCanIso}
\iota_{\Lambda}(\alpha,\beta):\Det_{\Lambda}C_{2}^{\bullet}\isocan\Det_{\Lambda}C_{1}^{\bullet}\tenseur_{\Lambda}\Det_{\Lambda}C^{\bullet}_{3}
\end{equation}
compatible with base-change whenever
\begin{equation}\nonumber
\suiteexacte{\alpha}{\beta}{C_{1}^{\bullet}}{C_{2}^{\bullet}}{C_{3}^{\bullet}}
\end{equation}
is a short exact sequence of complexes. If $\Lambda$ is reduced, $\Det_{\Lambda}(-)$ further extends to the derived category of perfect complexes of $\Lambda$-modules with morphisms restricted to quasi-isomorphisms and \eqref{EqCanIso} extends to distinguished triangles.
\subsubsection*{\Nekovar-Selmer Complexes}
Let $\Q\subset K\subset \Qbar$ be a finite extension of $\Q$ with ring of integers $\Ocal_{K}$ and let $S$ be a finite set of finite primes of $\Ocal_{K}$ containing the set $S_{p}$ of primes of $\Ocal_{K}$ above $p$. Let $A$ be a finite $p$-torsion ring or a complete local noetherian ring and let $M$ be an admissible $A[G_{K,S}]$-module in the sense of \cite[Section (3.2)]{SelmerComplexes} (when $A$ is a finite ring, this simply means that $M$ is a finite $p$-torsion module with a continuous action of $G_{K,S}$). For $G$ equal to $G_{K,S}$ or $G_{K_{v}}$ for a finite $v$, denote by $C^{\bullet}(G,M)$ the complex of continuous cochains with values in $M$.

A local condition at $v\in S$ is a pair $(C^{\bullet}_{?}(G_{K_{v}},M),i_{v})$ where $C^{\bullet}_{?}(G_{K_{v}},M)$ is a bounded complex and $i_{v}:C^{\bullet}_{?}(G_{K_{v}},M)\fleche C^{\bullet}(G_{K_{v}},M)$ is a morphism of complexes. Denote by $$i:C^{\bullet}(G_{K,S},M)\fleche \sommedirecte{v\in S}{}C^{\bullet}(G_{K_{v}},M)$$ the direct sum of the localization maps at $S$ and by $\iota$ the map
\begin{equation}\nonumber
i-\sommedirecte{v\in S}{}i_{v}:C^{\bullet}(G_{K,S},M)\oplus\sommedirecte{v\in S}{}C^{\bullet}_{?}(G_{K_{v}},M)\overset{}{\fleche}\sommedirecte{v\in S}{}C^{\bullet}(G_{K_{v}},M).
\end{equation}
The \Nekovar-Selmer complex $\RGamma_{?}(G_{K,S},M)$ of $M$ (see \cite[Chapter 6]{SelmerComplexes}) attached to the local conditions $(C^{\bullet}_{?}(G_{K_{v}},M),i_{v})$ for $v\in S$ is the complex
\begin{equation}\label{EqDefSelmer}
\Cone\left(C^{\bullet}(G_{K,S},M)\oplus\sommedirecte{v\in S}{}C^{\bullet}_{?}(G_{K_{v}},M)\overset{\iota}{\fleche}\sommedirecte{v\in S}{}C^{\bullet}(G_{K_{v}},M)\right)[-1]
\end{equation}
seen in the derived category. In a slight abuse of notations, we henceforth do not distinguish complexes and their images in the derived category so that we write $\RGamma_{?}(G_{K_{v}},M)$ for $C^{\bullet}_{?}(G_{K_{v}},M)$ and likewise in all similar situations. We also systematically assume that $(C^{\bullet}_{?}(G_{K_{v}},M),i_{v})$ is equal to $(C^{\bullet}(G_{K_{v}},M),\Id_{v})$ for all $v\in S_{p}$.

When $\RGamma_{?}(G_{K_{v}},T)$ is the zero complex for all $v\in S-S_{p}$, the attached \Nekovar-Selmer complex is  the complex of cohomology with compact support outside $p$
\begin{equation}\nonumber
\RGamma_{c}(G_{K,S},M)=\Cone\left(\RGamma(G_{K,S},M){\fleche}\sommedirecte{v\in S\backslash S_{p}}{}\RGamma(G_{K_{v}},M)\right)[-1].
\end{equation}
It is also equal to the complex $\RGamma_{c}(\Ocal_{K}[1/S],M)$ of étale cohomology with compact support outside $p$ (we emphasize that contrary to common practices, we impose no condition at primes above $p$). In the opposite direction, when $\RGamma_{?}(G_{K_{v}},M)$ is equal to $\RGamma(G_{K_{v}},T)$ and $i_{v}$ is the identity for all $v\in S$, the \Nekovar-Selmer complex is the complex $\RGamma(G_{K,S},M)$ of continuous cochains with values in $M$. 

When $M$ is nearly-ordinary in the sense that there exists a short exact sequence of non-zero $A[G_{K_{v}}]$-modules
\begin{equation}\nonumber
0\fleche M_{v}^{+}\fleche M\fleche M_{v}^{-}\fleche0
\end{equation}
for all $v|p$, we denote by $\RGamma_{f}(G_{K,S},M)$ the complex attached to the unramified condition $\RGamma(G_{K_{v}},M^{I_{v}})$ with its natural map to $\RGamma(G_{K_{v}},M)$ at $v\nmid p$ and to the nearly-ordinary condition $\RGamma(G_{K_{v}},M_{v}^{+})$ (with its natural inclusion in $\RGamma(G_{K_{v}},M)$) at $v|p$. Explicitly, this is the complex
\begin{equation}\nonumber
\Cone\left(\RGamma(G_{K,S},M)\oplus\sommedirecte{v\in S\backslash S_{p}}{}\RGamma(G_{K_{v}}/I_{v},M^{I_{v}})\oplus\sommedirecte{v|p}{}\RGamma(G_{K_{v}},M_{v}^{+})\fleche\sommedirecte{v\in S\backslash S_{p}}{}\RGamma(G_{K_{v}},M)\right)[-1].
\end{equation}
The cohomology of $\RGamma_{f}(G_{K,S},M)$ is denoted $H^{i}_{f}(G_{K,S},M)$.
\def\Dbar{\leavevmode\lower.6ex\hbox to 0pt{\hskip-.23ex \accent"16\hss}D}
  \def\cfac#1{\ifmmode\setbox7\hbox{$\accent"5E#1$}\else
  \setbox7\hbox{\accent"5E#1}\penalty 10000\relax\fi\raise 1\ht7
  \hbox{\lower1.15ex\hbox to 1\wd7{\hss\accent"13\hss}}\penalty 10000
  \hskip-1\wd7\penalty 10000\box7}
  \def\cftil#1{\ifmmode\setbox7\hbox{$\accent"5E#1$}\else
  \setbox7\hbox{\accent"5E#1}\penalty 10000\relax\fi\raise 1\ht7
  \hbox{\lower1.15ex\hbox to 1\wd7{\hss\accent"7E\hss}}\penalty 10000
  \hskip-1\wd7\penalty 10000\box7} \def\Dbar{\leavevmode\lower.6ex\hbox to
  0pt{\hskip-.23ex \accent"16\hss}D}
  \def\cfac#1{\ifmmode\setbox7\hbox{$\accent"5E#1$}\else
  \setbox7\hbox{\accent"5E#1}\penalty 10000\relax\fi\raise 1\ht7
  \hbox{\lower1.15ex\hbox to 1\wd7{\hss\accent"13\hss}}\penalty 10000
  \hskip-1\wd7\penalty 10000\box7}
  \def\cftil#1{\ifmmode\setbox7\hbox{$\accent"5E#1$}\else
  \setbox7\hbox{\accent"5E#1}\penalty 10000\relax\fi\raise 1\ht7
  \hbox{\lower1.15ex\hbox to 1\wd7{\hss\accent"7E\hss}}\penalty 10000
  \hskip-1\wd7\penalty 10000\box7} \def\Dbar{\leavevmode\lower.6ex\hbox to
  0pt{\hskip-.23ex \accent"16\hss}D}
  \def\cfac#1{\ifmmode\setbox7\hbox{$\accent"5E#1$}\else
  \setbox7\hbox{\accent"5E#1}\penalty 10000\relax\fi\raise 1\ht7
  \hbox{\lower1.15ex\hbox to 1\wd7{\hss\accent"13\hss}}\penalty 10000
  \hskip-1\wd7\penalty 10000\box7}
  \def\cftil#1{\ifmmode\setbox7\hbox{$\accent"5E#1$}\else
  \setbox7\hbox{\accent"5E#1}\penalty 10000\relax\fi\raise 1\ht7
  \hbox{\lower1.15ex\hbox to 1\wd7{\hss\accent"7E\hss}}\penalty 10000
  \hskip-1\wd7\penalty 10000\box7}
\providecommand{\bysame}{\leavevmode\hbox to3em{\hrulefill}\thinspace}
\providecommand{\MR}{\relax\ifhmode\unskip\space\fi MR }
\providecommand{\MRhref}[2]{%
  \href{http://www.ams.org/mathscinet-getitem?mr=#1}{#2}
}
\providecommand{\href}[2]{#2}


\begin{thebibliography}{EPW06}

\bibitem[AS97]{AshStevens}
Avner Ash and Glenn Stevens, \emph{{$p$}-adic deformations of cohomology
  classes of subgroups of {${\rm GL}(n,{\bf Z})$}}, Collect. Math. \textbf{48}
  (1997), no.~1-2, 1--30, Journ\'ees Arithm\'etiques (Barcelona, 1995).

\bibitem[BC09]{BellaicheChenevier}
Jo{\"e}l Bella{\"{\i}}che and Ga{\"e}tan Chenevier, \emph{Families of {G}alois
  representations and {S}elmer groups}, Ast\'erisque (2009), no.~324, xii+314.

\bibitem[BF96]{BurnsFlachMotivic}
David Burns and Matthias Flach, \emph{Motivic ${L}$-functions and {G}alois
  module structures}, Math. Annalen \textbf{305} (1996), 65--102.

\bibitem[BK90]{BlochKato}
Spencer Bloch and Kazuya Kato, \emph{{$L$}-functions and {T}amagawa numbers of
  motives}, The Grothendieck Festschrift, Vol.\ I, Progr. Math., vol.~86,
  Birkh\"auser Boston, Boston, MA, 1990, pp.~333--400.

\bibitem[BS07]{BreuilSchneider}
Christophe Breuil and Peter Schneider, \emph{First steps towards {$p$}-adic
  {L}anglands functoriality}, J. Reine Angew. Math. \textbf{610} (2007),
  149--180.

\bibitem[Che04]{ChenevierFamilles}
Ga{\"e}tan Chenevier, \emph{Familles {$p$}-adiques de formes automorphes pour
  {${\rm GL}_n$}}, J. Reine Angew. Math. \textbf{570} (2004), 143--217.

\bibitem[Che14]{ChenevierDeterminant}
\bysame, \emph{The p-adic analytic space of pseudocharacters of a profinite
  group, and pseudorepresentations over arbitrary rings}, Automorphic forms and
  {G}alois representations (Durham, July 2011), London Math. Soc. Lecture Note
  Ser., Cambridge Univ. Press, Cambridge, 2014.

\bibitem[CM98]{ColemanMazur}
Robert Coleman and Barry Mazur, \emph{The eigencurve}, Galois representations
  in arithmetic algebraic geometry (Durham, 1996), London Math. Soc. Lecture
  Note Ser., vol. 254, Cambridge Univ. Press, Cambridge, 1998, pp.~1--113.

\bibitem[Del77]{SGA41/2}
P.~Deligne, \emph{Cohomologie \'etale}, Lecture Notes in Mathematics, Vol. 569,
  Springer-Verlag, Berlin, 1977, S{\'e}minaire de G{\'e}om{\'e}trie
  Alg{\'e}brique du Bois-Marie SGA 4${1{\o}er 2}$, Avec la collaboration de J.
  F. Boutot, A. Grothendieck, L. Illusie et J. L. Verdier.

\bibitem[Del87]{DeligneDeterminant}
Pierre Deligne, \emph{Le d\'eterminant de la cohomologie}, Current trends in
  arithmetical algebraic geometry ({A}rcata, {C}alif., 1985), Contemp. Math.,
  vol.~67, Amer. Math. Soc., Providence, RI, 1987, pp.~93--177.

\bibitem[Dix68]{DixExposes}
\emph{Dix expos\'es sur la cohomologie des sch\'emas}, Advanced Studies in Pure
  Mathematics, Vol. 3, North-Holland Publishing Co., Amsterdam, 1968.

\bibitem[EK01]{EmertonKisin}
Matthew Emerton and Mark Kisin, \emph{Unit {$L$}-functions and a conjecture of
  {K}atz}, Ann. of Math. (2) \textbf{153} (2001), no.~2, 329--354.

\bibitem[Eme06]{EmertonInterpolationEigenvalues}
Matthew Emerton, \emph{On the interpolation of systems of eigenvalues attached
  to automorphic {H}ecke eigenforms}, Invent. Math. \textbf{164} (2006), no.~1,
  1--84.

\bibitem[EPW06]{EmertonPollackWeston}
Matthew Emerton, Robert Pollack, and Tom Weston, \emph{Variation of {I}wasawa
  invariants in {H}ida families}, Invent. Math. \textbf{163} (2006), no.~3,
  523--580. \MR{2207234 (2007a:11059)}

\bibitem[FK12]{FukayaKatoSharifi}
Takako Fukaya and Kazuya Kato, \emph{On conjectures of {S}harifi}, Preprint,
  2012.

\bibitem[FO12]{FouquetOchiai}
Olivier Fouquet and Tadashi Ochiai, \emph{Control theorems for {S}elmer groups
  of nearly ordinary deformations}, J. Reine Angew. Math. \textbf{666} (2012),
  163--187. \MR{2920885}

\bibitem[Fon92]{FontaineValeursSpeciales}
Jean-Marc Fontaine, \emph{Valeurs sp\'eciales des fonctions {$L$} des motifs},
  Ast\'erisque (1992), no.~206, Exp.\ No.\ 751, 4, 205--249, S{\'e}minaire
  Bourbaki, Vol. 1991/92.

\bibitem[Fou13]{HeckeETNC}
Olivier Fouquet, \emph{The {E}quivariant {T}amagawa {N}umber {C}onjecture for
  {M}odular {M}otives with coefficients in the {H}ecke algebra}, Preprint
  available at \texttt{http://www.math.u-psud.fr/$\sim$fouquet/}, 2013.

\bibitem[FPR94]{FontainePerrinRiou}
Jean-Marc Fontaine and Bernadette Perrin-Riou, \emph{Autour des conjectures de
  {B}loch et {K}ato: cohomologie galoisienne et valeurs de fonctions {$L$}},
  Motives (Seattle, WA, 1991), Proc. Sympos. Pure Math., vol.~55, Amer. Math.
  Soc., Providence, RI, 1994, pp.~599--706.

\bibitem[Gre89]{GreenbergIwasawaRepresentation}
Ralph Greenberg, \emph{Iwasawa theory for {$p$}-adic representations},
  Algebraic number theory, Adv. Stud. Pure Math., vol.~17, Academic Press,
  Boston, MA, 1989, pp.~97--137.

\bibitem[Gre91]{GreenbergIwasawaMotives}
\bysame, \emph{Iwasawa theory for motives}, $L$-functions and arithmetic
  (Durham, 1989), London Math. Soc. Lecture Note Ser., vol. 153, Cambridge
  Univ. Press, Cambridge, 1991, pp.~211--233.

\bibitem[GV00]{GreenbergVatsal}
Ralph Greenberg and Vinayak Vatsal, \emph{On the {I}wasawa invariants of
  elliptic curves}, Invent. Math. \textbf{142} (2000), no.~1, 17--63.

\bibitem[Hid86]{HidaInventionesOrdinary}
Haruzo Hida, \emph{Galois representations into {${\rm GL}\sb 2({\bf Z}\sb
  p[[X]])$} attached to ordinary cusp forms}, Invent. Math. \textbf{85} (1986),
  no.~3, 545--613.

\bibitem[Ill94]{IllusieMonodromie}
Luc Illusie, \emph{Autour du th\'eor\`eme de monodromie locale}, Ast\'erisque
  (1994), no.~223, 9--57, P{\'e}riodes $p$-adiques (Bures-sur-Yvette, 1988).
  \MR{1293970 (95k:14032)}

\bibitem[Kat93a]{KatoHodgeIwasawa}
Kazuya Kato, \emph{Iwasawa theory and {$p$}-adic {H}odge theory}, Kodai Math.
  J. \textbf{16} (1993), no.~1, 1--31.

\bibitem[Kat93b]{KatoViaBdR}
\bysame, \emph{Lectures on the approach to {I}wasawa theory for {H}asse-{W}eil
  {$L$}-functions via {$B\sb {\rm dR}$}. {I}}, Arithmetic algebraic geometry
  (Trento, 1991), Lecture Notes in Math., vol. 1553, Springer, Berlin, 1993,
  pp.~50--163.

\bibitem[Kat99]{KatoReciprocity}
\emph{Generalized explicit reciprocity laws}, 57--126.

\bibitem[Kat04]{KatoEuler}
Kazuya Kato, \emph{{$p$}-adic {H}odge theory and values of zeta functions of
  modular forms}, Ast\'erisque (2004), no.~295, ix, 117--290, Cohomologies
  $p$-adiques et applications arithm\'etiques. III.

\bibitem[Kit94]{KitagawaLpadic}
Koji Kitagawa, \emph{On standard {$p$}-adic {$L$}-functions of families of
  elliptic cusp forms}, $p$-adic monodromy and the Birch and Swinnerton-Dyer
  conjecture (Boston, MA, 1991), Contemp. Math., vol. 165, Amer. Math. Soc.,
  Providence, RI, 1994, pp.~81--110.

\bibitem[KM76]{MumfordKnudsen}
Finn~Faye Knudsen and David Mumford, \emph{The projectivity of the moduli space
  of stable curves. {I}. {P}reliminaries on ``det'' and ``{D}iv''}, Math.
  Scand. \textbf{39} (1976), no.~1, 19--55.

\bibitem[Maz79]{MazurValues}
Barry Mazur, \emph{On the arithmetic of special values of {$L$} functions},
  Invent. Math. \textbf{55} (1979), no.~3, 207--240.

\bibitem[Maz89]{MazurDeformation}
\bysame, \emph{Deforming {G}alois representations}, Galois groups over ${\bf
  Q}$ (Berkeley, CA, 1987), Math. Sci. Res. Inst. Publ., vol.~16, Springer, New
  York, 1989, pp.~385--437.

\bibitem[Nek06]{SelmerComplexes}
Jan Nekov{\'a}{\v{r}}, \emph{Selmer complexes}, Ast\'erisque (2006), no.~310,
  559.

\bibitem[Och06]{OchiaiMainConjecture}
Tadashi Ochiai, \emph{On the two-variable {I}wasawa main conjecture}, Compos.
  Math. \textbf{142} (2006), no.~5, 1157--1200.

\bibitem[PR95]{PerrinRiouLpadique}
Bernadette Perrin-Riou, \emph{Fonctions {$L$} {$p$}-adiques des
  repr\'esentations {$p$}-adiques}, Ast\'erisque (1995), no.~229, 198.

\bibitem[Sah14]{TheseSaha}
Jyoti~Prakash Saha, \emph{An algebraic $p$-adic $l$-function for ordinary
  families. {T}hèse de l'{U}niversité {P}aris {XI}}, 2014.

\bibitem[Sch13]{ScholzeTorsion}
Peter Scholze, \emph{On torsion in the cohomology of locally symmetric
  varieties}, Available at \texttt{http://www.math.uni-bonn.de/people/scholze},
  2013.

\bibitem[Tat66]{TateBSD}
John Tate, \emph{On the conjectures of {B}irch and {S}winnerton-{D}yer and a
  geometric analog}, S\'eminaire {B}ourbaki, {V}ol.\ 9, Soc. Math. France,
  Paris, 1966, pp.~Exp.\ No.\ 306, 415--440. \MR{1610977}

\bibitem[Urb11]{UrbanEigenvarieties}
Eric Urban, \emph{Eigenvarieties for reductive groups}, Ann. of Math. (2)
  \textbf{174} (2011), no.~3, 1685--1784. \MR{2846490}

\bibitem[Wil88]{WilesOrdinaryLambdaAdic}
Andrew Wiles, \emph{On ordinary {$\lambda$}-adic representations associated to
  modular forms}, Invent. Math. \textbf{94} (1988), no.~3, 529--573.

\end{thebibliography}
\end{document}